\documentclass[a4paper,12pt]{article}
\usepackage{authblk}

\usepackage{mlmodern}
\usepackage[T1]{fontenc}

\usepackage{comment}
\usepackage{enumitem}
\usepackage[british,UKenglish,USenglish,american]{babel}
\usepackage{amsmath,amsfonts,amssymb,amsthm}
\usepackage[dvipsnames]{xcolor}
\usepackage[colorlinks=true, linkcolor=blue, citecolor=violet]{hyperref}
\usepackage{graphics}
\usepackage{makeidx}
\usepackage{fancyhdr}
\usepackage[dvipsnames]{xcolor}
\numberwithin{equation}{section}
\usepackage[margin=0.9in]{geometry}
\usepackage{soul}

\newtheorem{theorem}{Theorem}[section]
\newtheorem{proposition}[theorem]{Proposition}

\newtheorem{lemma}[theorem]{Lemma}
\newtheorem{prop}[theorem]{Proposition}
\newtheorem{definition}[theorem]{Definition}
\newtheorem{remark}[theorem]{Remark}
\newtheorem{assumptions}[theorem]{Assumptions}
\newcommand{\E}{\noindent{$\mathbb{E}$ \ }}

\newcommand{\baonew}[1]{{#1}}
\newcommand{\bao}[1]{{#1}}
\newcommand{\stef}[1]{{\color{blue} #1}}

\newcommand{\bra}[1]{\left(#1\right)}
\newcommand{\tto}{\theta_{-t}\omega}
\newcommand{\LO}[1]{L^{#1}(D)}

\def\R{\mathbb{R}}
\def\C{\mathbb{C}}
\def\N{\mathbb{N}}

\def\P{\mathbb{P}}
\def\E{\mathbb{E}}

\def\cA{\mathcal{A}}

\def\cD{\mathcal{D}}

\def\cF{\mathcal{F}}

\def\cL{\mathcal{L}}

\def\cX{\mathcal{X}}	
\def\cY{\mathcal{Y}}

\def\txtd{{\textnormal{d}}}

\def\txtD{{\textnormal{D}}}

\newcommand{\norm}[1]{\left\lVert #1 \right\rVert}
\newcommand{\abs}[1]{\left| #1 \right|}

\date{}

\allowdisplaybreaks

\title{Pathwise mild solutions for superlinear stochastic evolution equations and their attractors}

\author[1]{Alexandra Blessing (Neam\c{t}u)}
\author[1]{Tim Seitz}
\author[2]{Stefanie Sonner}
\author[3]{Bao Quoc Tang}
\affil[1]{\small  Department of Mathematics and Statistics, University of Konstanz, Germany.\break
\href{mailto:alexandra.neamtu@uni-konstanz.de}{alexandra.blessing@uni-konstanz.de}, \href{tim.seitz@uni-konstanz.de}{tim.seitz@uni-konstanz.de}}
\affil[2]{\small Radboud University, IMAPP - Mathematics, The Netherlands\break   
\href{stefanie.sonner@ru.nl}{stefanie.sonner@ru.nl}}
\affil[3]{\small Department of Mathematics and Scientific Computing, University of Graz, Austria\break  
\href{mailto:quoc.tang@uni-graz.at}{quoc.tang@uni-graz.at}}
\begin{document}
     
\maketitle
        
     \begin{abstract} We investigate stochastic parabolic evolution equations with time-dependent random generators and locally Lipschitz continuous drift terms.~Using pathwise mild solutions, we construct an infinite-dimensional stationary Ornstein-Uhlenbeck type process, which is shown to be tempered in suitable function spaces.~This property, together with a bootstrapping argument based on the regularizing effect of parabolic evolution families, is then applied to prove the global well-posedness and the existence of a  random attractor for reaction-diffusion equations with random {non-autonomous} generators and nonlinearities  satisfying certain growth and dissipativity assumptions.
     \end{abstract}

\tableofcontents
        
\section{Introduction}

We aim to investigate the well-posedness and long-time behavior of solutions of parabolic superlinear stochastic evolution equations with additive noise of the form 
		\begin{align}\label{eq:MainEq}
			\begin{cases}
                \txtd u(t) &= [A(t,\omega)u(t)+F(u(t)) + {f}]~\txtd t + \sigma \txtd W_t\\
                u(0)&=u_0(\omega),
            \end{cases}
		\end{align}
		where $(W_t)_{t\geq 0}$ is an infinite-dimensional Brownian motion on a filtered probability space $(\overline{\Omega},\mathcal{\overline{F}},\mathbb{\overline{P}}, (\mathcal{F}_t)_{t\geq 0})$, $\sigma>0$ denotes the intensity of the noise,
  $(A(t,\omega))_{t\in \R,\,\omega\in \overline{\Omega}}$ is a family of random, time-dependent operators, $F$ a nonlinear drift term, {$f$ a given external force}, and the initial datum $u_0$ is $\mathcal{F}_0$-measurable.
More specifically, we establish the local well-posedness for such problems and study properties of an  {Ornstein-Uhlenbeck type} process in an abstract setting. These results are subsequently applied to prove the existence of a random attractor for reaction-diffusion equations with random generators under suitable growth and dissipativity assumptions on the nonlinearity. 

Several dynamical aspects have been investigated for such random evolution equations without noise, i.e.~for problem \eqref{eq:MainEq} with $\sigma=0$. For instance, principal Lyapunov exponents and Floquet theory were analyzed in~\cite{MShen1,book,MShen2}, stable and unstable manifolds in~\cite{CDLS} and multiplicative ergodic theorems in~\cite{CDLS,lns}. 
The existence of random attractors for problem \eqref{eq:MainEq} with $\sigma>0$ was shown in \cite{kuehn2021random}, however under {additional regularity assumptions on the noise} and global Lipschitz continuity of the nonlinearity $F$.~Here, we significantly generalize these results  and investigate the existence of random attractors for drift terms $F$ that are only locally Lipschitz continuous. To this end, as in \cite{kuehn2021random}, we use the concept of pathwise mild solutions. \\
  
 In order to capture the asymptotic behavior of~\eqref{eq:MainEq} we rely on the random dynamical systems approach.~This requires, like in the references mentioned above, the following structural assumption for the random generators,
\[ 
A(t,\omega):=A(\theta_t\omega),~\text{ for all } t\in \R \text{ and }\omega\in\Omega, 
\]
where $(\Omega,\mathcal{F},\mathbb{P},(\theta_t)_{t\in\R})$ is an ergodic metric dynamical system.~This property ensures that the random parabolic evolution family $(U(t,s,\omega))_{t\geq s,\omega\in\Omega}$ generated by the operators
$(A(t,\omega))_{t\in \R,\,\omega\in \Omega}$
forms a random dynamical system. For further details, we refer to Subsection~\ref{sd}. The mild solution of~\eqref{eq:MainEq} should be given by the variation of constants formula
 \begin{align*}
    u(t)=U(t,0,\omega)u_0+ \int_0^t U(t,s,\omega) {(F(u(s)) +f )}~\txtd s+ \sigma \int_0^t U(t,s,\omega)~\txtd W_s(\omega).
        \end{align*}
However, the stochastic integral is not well-defined due to measurability issues.~This problem can be resolved by formally applying integration by parts, which leads to the following pathwise representation formula
        \begin{align}\label{eq:defpwms}
        \begin{split}
            u(t)&= U(t,0,\omega)u_0 + \sigma U(t,0,\omega)W_t(\omega) + \int_0^t U(t,s,\omega) {(F(u(s)) +f )}~\txtd s \\
            & -  \sigma\int_0^t U(t,s,\omega) A(s,\omega)(W_t(\omega)-W_s(\omega))~\txtd s.
            \end{split}
        \end{align}
A process satisfying \eqref{eq:defpwms} is called a pathwise mild solution of \eqref{eq:MainEq}.~Pathwise mild solutions were introduced in \cite{Pronk2014} and their existence and uniqueness investigated for stochastic parabolic equations of the form \eqref{eq:MainEq} with globally Lipschitz continuous nonlinearities $F$ (and more general noise terms).~These well-posedness results for pathwise mild solutions were applied in~\cite{kuehn2021random} to investigate the existence of global and exponential random pullback attractors for problem \eqref{eq:MainEq} with globally Lipschitz nonlinearity $F$.~Here, we relax these assumptions in the specific case of reaction-diffusion equations allowing for locally Lipschitz continuous nonlinearities with superlinear growth and more general noise, which significantly extends our previous results in~\cite{kuehn2021random}.~To the best of our knowledge,~\cite{kuehn2021random} is the only work that addresses random attractors for stochastic parabolic evolution equations of the form \eqref{eq:MainEq} using the concept of pathwise mild solutions. On the other hand, there is a vast literature on random attractors for stochastic reaction-diffusion equations driven by real-valued additive or linear multiplicative noise. We only mention a few references, e.g.~\cite{Lu,GuWang,Bao,cao,liu2017long}. These works heavily rely on the standard approach to prove the existence of random attractors, where a finite-dimensional stationary Ornstein-Uhlenbeck process is used to transform the stochastic problem into a partial differential equation with random, non-autonomous coefficients.~{For infinite-dimensional additive noise, the existence of random attractors for reaction-diffusion equations has been obtained in~\cite{Wang} and the references specified therein using mean random dynamical systems, instead of a pathwise approach as presented here.} 
{For general criteria for the existence of random attractors for stochastic partial differential equations with infinite-dimensional additive and finite-dimensional linear multiplicative noise based on a variational approach we refer to~\cite{Gess2011,Gess1,Gess2,Andre}.} 

\medskip
In conclusion, the main novelties of this work are two-fold: firstly, we consider stochastic reaction-diffusion equations in a bounded domain $D\subset\R^N$ with time-dependent random operators and infinite-dimensional noise and secondly, we work {with} {nonlinearities that are only locally Lipschitz continuous}. Furthermore, different from classical approaches, we use the framework of pathwise mild solutions introduced in~\cite{Pronk2014}.~This leads to several technical difficulties. For example, due to the time-dependent random differential operators and the only locally Lipschitz continuous nonlinearity we cannot directly obtain a-priori estimates for the solution in more regular functional spaces, more specifically, in the fractional power spaces corresponding to the random non-autonomous generators. Such estimates are essential for the global well-posedness of~\eqref{eq:MainEq} and the compactness argument, which is needed for the existence of the random attractor.~In order to overcome this issue, we derive a-priori bounds in  suitable $L^{\rho+1}$-spaces where $\rho$ is determined by the subcritical growth of the nonlinearity $F$. These bounds combined with a bootstrapping argument based on the regularizing property of parabolic evolution families will entail a-priori estimates of the solution in more regular spaces, similarly as in~\cite{Caraballo} for deterministic reaction-diffusion equations with non-autonomous generators.~This technique heavily relies on the construction of a stationary infinite-dimensional Ornstein-Uhlenbeck process in Section~\ref{ou}, which is shown to be tempered in $L^2$ as well as in the fractional power spaces corresponding to  the random non-autonomous generators.~Furthermore, Sobolev embeddings imply the temperedness of this process in $L^q$, where $q$ depends on the dimension $N$ of the spatial domain $D$.~Such tools are not required to prove the existence of random attractors in simpler settings, for instance for reaction-diffusion equations with finite dimensional noise as considered in~\cite{Lu,GuWang,Bao,cao}.~{Moreover, the case of random non-autonomous generators is not covered by the references mentioned above.}~As a by-product, our work also extends the existence results for pullback attractors in~\cite{Caraballo} to a random setting with random non-autonomous generators and infinite-dimensional additive noise. \\ 
        
We finally mention recent works addressing random attractors for stochastic parabolic partial differential equations driven by nonlinear multiplicative noise {that use a pathwise construction of the solutions based on rough path theory~\cite{Lin2023,BS} or fractional calculus~\cite{fractional}}, however, the nonlinearity $F$  is assumed to be globally Lipschitz.~An extension of the techniques employed here to nonlinear multiplicative noise will be pursued in future works.\\ 

We expect that a possible future application of our results are stochastic parabolic partial differential equations on random moving domains, as considered in~\cite{djurdjevac2021linear}. However, a suitable well-posedness theory for such problems is still lacking.~After transforming the problem into a random partial differential equation on a fixed domain, the differential operators should have the same structure as in~\eqref{eq:MainEq}.~Therefore, similar measurability issues should occur in the classical definition of mild solutions, which we resolve here using pathwise mild solutions.~We finally mention that attractors for stochastic reaction-diffusion equations on time-evolving (not random) domains were investigated, for instance in~\cite{crauelkloeden}.\\

Our paper is structured as follows.~In Section~\ref{sec:p} we provide the required background on random dynamical systems and summarize preliminary results on random parabolic evolution operators and pathwise mild solutions.~In Section~\ref{sec:loc:sol} we show the local well-posedness of the abstract problem~\eqref{eq:MainEq} under suitable growth restrictions on the nonlinearity. 
In Section~\ref{ou} a stationary infinite-dimensional Ornstein-Uhlenbeck type process is constructed and the temperedness of this process is shown in suitable function spaces.
We then apply these results to study random attractors of reaction-diffusion equations with non-autonomous random differential operators and additive noise. In Section~\ref{global} we show   via a bootstrapping argument that the local solutions obtained in Section \ref{sec:loc:sol} are global if $F$ satisfies an additional dissipativity assumption. This is done at the level of a partial differential equation with random non-autonomous coefficients obtained from~\eqref{eq:MainEq} by subtracting the stationary infinite-dimensional Ornstein-Uhlenbeck type process.~This is essential for the estimates performed in Section~\ref{main} to show the existence of a random attractor.~Finally, we generalize the results for problems with higher-order uniformly elliptic random differential operators in Section~\ref{ho}.

\subsection*{Acknowledgements.} A. Blessing was supported by the DFG grant 543163250.~A. Blessing acknowledges support from DFG CRC/TRR 388 {\em Rough Analysis, Stochastic Dynamics and Related Fields}, Project A06. 

B.Q. Tang is supported by 
the FWF project number I-5213 (Quasi-steady-state approximation for PDE). Part of this work is completed during the research stay of B.Q. Tang at the Vietnam Institute for Advanced Study in Mathematics (VIASM) within the project \textit{Nonlinear problems in physics and geometry}. The institute's generous support and hospitality is greatly acknowledged.

\section{Preliminaries}\label{sec:p}
		

        \subsection{Random dynamical systems}\label{sd}
        Let {$\mathcal X$} be a separable, reflexive Banach space of type 2 
         and $(\Omega,\mathcal{F},\P)$ be a probability space.~As well-known examples of type 2 Banach spaces we mention $L^p(D)$ for $p\geq 2$, where $D\subset \R^N$, for $N\geq 1$, is a  bounded domain. To introduce an appropriate model for the noise we first recall the notions of metric and random dynamical systems~\cite{Arnold1998}.
			\begin{definition}\label{def:MDS}
				If $(\theta_t)_{t\in \R}$ is a family of mappings $\theta_t:\Omega\to\Omega$ such that for every $A\in \mathcal{F}$ and $t\in \R$ we have $\P(\theta_t^{-1} A)=\P(A)$, then $(\theta_t)_{t\in \R}$ is called measure-preserving. 
				
				The quadruple $(\Omega,\mathcal{F},\P,(\theta_t)_{t\in \R})$ is called a metric dynamical system if
				\begin{itemize}
					\itemsep -4pt
					\item[i)] $\theta_0=\textnormal{Id}_\Omega$,
					\item[ii)] $(t,\omega)\mapsto \theta_t\omega$ is $\mathcal{B}(\R)\otimes \mathcal{F}-\mathcal{F}$ measurable,
					\item[iii)] and $\theta_{t+s}=\theta_t\circ \theta_s$ holds for all $t,s \in \R$. 
				\end{itemize}
				We call it an ergodic metric dynamical system if for any $A\in \mathcal{F}$, which is $(\theta_t)_{t\in \R}$-invariant, we have $\P(A)\in \{0,1\}$.
			\end{definition}
			Similar to \cite{kuehn2021random} we want to introduce a metric dynamical system associated with a two-sided $\cX$-valued {Brownian motion defined on a filtered probability space $(\overline{\Omega},\overline{\cF},\overline{\P},(\cF_t)_{t\geq 0})$, where  $(\cF_t)_{t\geq 0}$ is the natural filtration of the Brownian motion.}
   To this end, let $H$ be a Hilbert space, $(W_H(t))_{t\geq 0}$ be an $H$-cylindrical Brownian motion and $G:H\to \cX$ be a $\gamma$-radonifying operator. This means that for any sequence $(\gamma_n)_{n\in \N}$ of independent Gaussian random variables, the series
			\begin{align}\label{eq:GammaRadonifying}
				\E\left[\bigg\|\sum_{n\in \N}\gamma_n Ge_n\bigg\|^2_{\cX}\right]<\infty
			\end{align}
			is finite, where $(e_n)_{n\in \N}$ is an orthonormal basis in $H$. If $H$ is isomorphic to $\cX$, then \eqref{eq:GammaRadonifying} implies that $G\in \cL_2(H)$, where $\cL_2(H)$ denotes the space of Hilbert-Schmidt operators. In this case $\|G\|_{\cL_2(H)}=\text{Tr}(GG^*)$. According to \cite[Proposition 8.8]{ISEM}, the series
			\begin{align}\label{eq:BrownianMotion}
				\sum_{n\in \N} W_H(t)\tilde{e}_nGe_n
			\end{align}
			converges almost surely and defines an $\cX$-valued Brownian motion with covariance operator $tGG^*$, where $(\tilde{e}_n)_{n\in \N}\subset \big(\textrm{ker}(G)\big)^\bot$ is an orthonormal basis. {Let $(W^1_t)_{t\geq 0}$ and $(W^2_t)_{t\geq 0}$ be two independent Brownian motions. We define the two-sided Brownian motion with values in $\cX$ by}
			\begin{align*}
				W_t:=
				\begin{cases}
					W_1(t), & t\geq 0,\\
					W_2(-t), & t< 0.
				\end{cases}
			\end{align*} 
   Since we are working in the framework of random dynamical systems, we aim to introduce a canonical probability space associated to $(W_t)_{t\geq 0}$ and identify $W_t(\omega)=\omega_t$. Therefore, we denote by $C_0(\R;\cX)$ the subset of continuous functions which are zero in zero and equip it with the compact open topology.~Furthermore, we denote by $\P_W$ the Wiener measure. Then Kolmogorov's theorem yields the existence of a probability space $(C_0(\R;\cX),\mathcal{B}(C_0(\R;\cX)),\P_W)$ such that the Brownian motion is the canonical process $W_t(\omega):=\omega_t$ for $\omega\in C_0(\R;\cX)$. We introduce the Wiener shift by 
\begin{align*}
				\theta_t \omega(\cdot):=\omega(t+\cdot)-\omega(t),~\text{ for } t\in\R, \omega\in C_0(\R;\cX)
    \end{align*}
	and obtain the ergodic metric dynamical system $(C_0(\R;\cX),\mathcal{B}(C_0(\R;\cX)),\P_W,(\theta_t)_{t\in \R})$.
 

    In the next lemma, we collect standard growth properties of the noise that we will frequently use.
\begin{lemma}{\rm (\cite[Lemma 3.3]{Gess2011})}\label{lem:GrowthPropBM}
				There exists a $(\theta_t)_{t\in \R}$-invariant subset $\Omega\subset C_0(\R;\cX)$ of full measure, with the following properties:
\begin{itemize}
\itemsep -4pt
    \item[i)] For all $\omega \in \Omega$ and any $\varepsilon>0$ there exists a time $T_0(\varepsilon,\omega)>0$ such that for any $\abs{t}\geq T_0(\varepsilon,\omega)$ we have  linear growth of $\omega$, i.e. 
	\begin{align}\label{eq:GrowthBM}
		\norm{\omega(t)}_{\cX}\leq \varepsilon \abs{t}.
	\end{align}
    \item[ii)] For any  $\gamma\in(0,\frac{1}{2})$ and any $[s,r]\subset \R$ there exist a constant $c_{s,r,\gamma}(\omega)>0$ such that $c_{s,r,\gamma}(\cdot)\in L^1(\Omega)$ and 				\begin{align}\label{eq:HölderCondBM}
	\norm{\omega}_{C^\gamma([s,r];\cX)}\leq c_{s,r,\gamma}(\omega).
	\end{align}
\end{itemize}
\end{lemma}

\begin{remark}\label{rem:decaynoise}
Condition \eqref{eq:GrowthBM} implies that $\omega\in\Omega$ has subexponential growth, meaning that for any $\varepsilon>0$ there exists a time $t_0=t_0(\varepsilon,\omega)$ and a constant $c_\varepsilon(\omega,t_0)$ such that for all $\abs{t}\geq t_0$, we have
\begin{align}\label{eq:SubExpBM}
					\norm{\omega(t)}_{\cX}\leq c_\varepsilon(\omega,t_0) e^{\varepsilon \abs{t}}.
				\end{align} 
\end{remark}
            
\begin{remark}\label{mds:final}
  From now on, we consider the set $\Omega$ satisfying the properties of Lemma \ref{lem:GrowthPropBM}. We equip this set with the trace $\sigma$-algebra $\mathcal{F}$ $:=\Omega\cap \mathcal{B}(C_0(\mathbb{R};\cX))$ and take the restriction $\P$ of $\mathbb{P}_W$. Again, by Komogorov’s theorem about the existence of a H\"older-continuous version, this set is contained in $C_0(\R;\cX)$, has full measure and is $(\theta_t)_{t\in\R}$-invariant. Moreover, the new quadruple which we denote by  $(\Omega,\mathcal{F},\mathbb{P},(\theta_t)_{t\in\R})$ forms again a metric dynamical system.
\end{remark}

         Next, we recall the definition of a random dynamical system and introduce tempered random sets.
            \begin{definition}
				A continuous random dynamical system (RDS) on $\cX$ over a metric dynamical system $(\Omega,\mathcal{F},\P,(\theta_t)_{t\in\R})$ is a mapping 
				\[\phi:\R^+\times \Omega\times \cX\to \cX,\quad  (t,\omega,x)\mapsto \phi(t,\omega,x),\]
				which is $(\mathcal{B}(\R^+)\otimes \mathcal{F}\otimes \mathcal{B}(\cX),\mathcal{B}(\cX))$-measurable and has the following properties:
				\begin{itemize}
					\itemsep -4pt
					\item[i)] $\phi(0,\omega,\cdot)=\textnormal{Id}_{\cX}$ for every $\omega\in \Omega$,
					\item[ii)] the (perfect) cocycle property holds, which means that for all $\omega\in \Omega$, $t,s\in \R^+$ and $x\in \cX$ we have
					\begin{align*}
		\phi(t+s,\omega,x)=\phi(t,\theta_s \omega,\phi(s,\omega,x)),
					\end{align*}
					\item[iii)] the map $\phi(t,\omega,\cdot):\cX\to \cX$ is continuous for every $t\in \R^+$ and $\omega\in \Omega$.
				\end{itemize}
			\end{definition}

We further introduce another separable reflexive Banach space of type $2$ denoted by $\cY$. This function space can coincide with $\cX$, but in our application in Section \ref{sec:attr} we consider random sets in spaces that are different from the phase space  $\cX$.

\begin{definition}\label{def:RandomSet}
				A multifunction $\{B(\omega)\}_{\omega\in \Omega}$ of nonempty closed subsets $B(\omega)\subset \cY$ is called a {random set in $\cY$} if
				\begin{align*}
					\omega \mapsto \inf_{y\in B(\omega)}\norm{x-y}_{\cY},
				\end{align*}
				is a random variable for any $x\in \cY$. If, in addition, $B(\omega)$ is bounded (compact) for every $\omega\in \Omega$, then $\{B(\omega)\}_{\omega\in \Omega}$ is called bounded (compact) random set.
\end{definition}
            
\begin{definition}\label{def:Tempered}
				A random variable $Y:\Omega\to (0,\infty)$ is called tempered from above with respect to $(\theta_t)_{t\in \R}$ if
				\begin{align*}
					\limsup_{t\to\pm \infty} \frac{\ln^+ Y(\theta_t\omega)}{t}=0~\text { for all }\omega\in\Omega,
				\end{align*}
                where $\ln^+ x:=\max\{\ln x ,0\}$.
            	The random variable $Y$ is called tempered from below if $1/Y$ is tempered from above. A random variable is tempered if and only if it is tempered from above and from below.		
             
             Moreover, if $\{B(\omega)\}_{\omega\in \Omega}$ is a bounded random set {in $\mathcal Y$} and $\omega \mapsto \sup_{x\in B(\omega)}\norm{x}_{{\mathcal Y}}$ is tempered, then $\{B(\omega)\}_{\omega\in \Omega}$ is called a tempered set {in $\mathcal{Y}$}. {We denote by $\mathcal{D}(\mathcal Y)$ the collection of all tempered sets in $\mathcal Y$.}

			\end{definition}
            \begin{remark}\label{rem:EquivTemp}
                Equivalently to Definition \ref{def:Tempered}, a random variable $Y:\Omega\to (0,\infty)$ is tempered if
                \begin{align*}
                    \lim\limits_{t\to \pm \infty} e^{-\gamma|t|} Y(\theta_{-t}\omega) =0
                \end{align*}
                 for every $\gamma>0$.
            \end{remark}
             \begin{remark}\label{rem:temp}
                 Using Remark \ref{rem:EquivTemp} it is easy to see that finite sums and positive powers of tempered random variables are  tempered. 
            \end{remark}
			A sufficient condition for the temperedness of a random variable is given in \cite[Theorem 4.1.3]{Arnold1998}, namely, 
			\begin{align}\label{eq:TempCond}
				\E\Big[\sup_{t\in [0,1]} Y(\theta_t \omega)\Big]<\infty.
			\end{align}
			Note that the interval $[0,1]$ is used for simplicity, but we can replace it by an arbitrary interval $[a,b]$, see \cite[Lemma 8]{Bessaih2014}. This follows from a Borell-Cantelli argument based on the ergodicity of $(\theta_t)_{t\in\R}$.

Finally, we introduce the notion of random absorbing sets and attractors and recall a classical result providing sufficient conditions for the existence of  random attractors. 
     
            \begin{definition}
                A random set $\{K(\omega)\}_{\omega\in\Omega} \in \mathcal D{(\mathcal Y)}$ is called a random absorbing set for the RDS $\phi$ {in $\mathcal Y$}, if for any random tempered set $\{B(\omega)\}_{\omega\in\Omega}\in \mathcal D{(\mathcal X)}$ there is a time $t_B(\omega)>0$ such that
                \begin{equation*}
                    \phi(t,\theta_{-t}\omega,B(\theta_{-t}\omega)) \subset K(\omega) \quad \forall t\ge t_B(\omega),
                \end{equation*}
                for $\mathbb{P}$-a.s. $\omega\in\Omega$.
            \end{definition}

            In Section~\ref{sec:attr} we show the existence of a random absorbing set in a suitable space $\cY$ for the random dynamical system generated by~\eqref{eq:MainEq} on the phase space $\cX$, see Lemma~\ref{lem1} and Lemma~\ref{lem2_1}. This justifies the choice of a space $\cY$ possibly different from $\cX$ in the definitions above.

\begin{definition}
A random set $\{\cA(\omega)\}_{\omega\in\Omega}$ of $\cX$ is called a random attractor for a RDS $\phi$ if for $\mathbb P$-a.s. $\omega\in\Omega$,
\begin{itemize}[align=left]                    \item[i)] $\cA(\omega)$ is compact in $\cX$;
\item[ii)] $\{\cA(\omega)\}_{\omega\in\Omega}$ is invariant under $\phi$, i.e.
\begin{equation*}
\phi(t,\omega,\cA(\omega)) = \cA(\theta_t\omega) \quad \forall t\ge 0;
                    \end{equation*}
\item[iii)] $\{\cA(\omega)\}_{\omega\in\Omega}$ is pullback attracting in $\cD{(\mathcal X)}$,~i.e. for any $\{B(\omega)\}_{\omega\in\Omega} \in \cD\bao{(\mathcal X)}$ it holds
                    \begin{equation*}
                        \lim_{t\to\infty}d_{H}(\phi(t,\theta_{-t}\omega,B(\theta_{-t}\omega)), \cA(\omega)) = 0
                    \end{equation*}
                    where $d_H$ is the Hausdorff semi-distance in $\cX$.
                \end{itemize}

            \end{definition}
            
\begin{theorem}{\em~(\cite[Theorem 3.5]{FlSchm})}\label{thm:attractor_general}
                Let $\phi$ be a continuous RDS in $\cX$. Assume that $\phi$ has a random compact absorbing set $\{B(\omega)\}_{\omega\in\Omega}$ in $\cX$. 
                Then $\phi$ possesses a unique random attractor $\{\cA(\omega)\}_{\omega\in\Omega}$ which is defined by
                \begin{equation*}
                    \cA(\omega) =\bigcap_{\tau \ge 0}\overline{\bigcup_{t\ge \tau}\phi(t,\theta_{-t}\omega,B(\theta_{-t}\omega))}.
                \end{equation*}
            \end{theorem}
      
        \subsection{Pathwise mild solutions}
        We consider the abstract parabolic problem \eqref{eq:MainEq}.
        Similar to~\cite{Pronk2014} and~\cite{kuehn2021random} we make the following assumptions for the non-autonomous random differential operators in order to ensure that $(A(t,\omega))_{t\in \R,\,\omega\in \Omega}$ generates a parabolic evolution family {on a separable, reflexive Banach space of type $2$ denoted by $X$}. 
        \begin{assumptions}\label{ass:OperatorFamily}\hfill
			\begin{itemize}
				\item[i)] The operators $A(t,\omega)$ are closed and densely defined with fixed domains, i.e.  $\mathcal{D}_A:=D(A(t,\omega))$ for every $t\in \R,\omega\in \Omega$. Furthermore, they have bounded imaginary powers, i.e. there exists $C>0$ such that
                \begin{align*}
                    \sup_{\abs{s}\leq 1} \norm{(-A(t,\omega))^{is}}_{\mathcal{L}(X)}\leq C
                \end{align*}
                for every $t,s\in \R, \omega\in \Omega.$
				\item[ii)] The operators are sectorial, i.e.~there exists  $\vartheta\in (\frac{\pi}{2},\pi)$, such that $\Sigma_\vartheta:=\{z\in \C~:~\abs{\textnormal{arg}(z)}<\vartheta\}\cup \{0\}\subset \rho(A(t,\omega))$ for all $t\in \R,\omega\in \Omega$, where $\rho$ denotes the resolvent, and there exists a constant $M>0$ such that
				\begin{align*}
					\norm{(z-A(t,\omega))^{-1}}_{\mathcal{L}(X)}\leq \frac{M}{1+\abs{z}}
				\end{align*}
				for every $t\in \R,\omega\in \Omega$ and $z\in \Sigma_\vartheta$.
				\item[iii)] There exist $\nu\in (0,1]$ and $C>0$ such that
				\begin{align}\label{kt}
					\norm{A(t,\omega)-A(s,\omega)}_{\mathcal{L}(\mathcal{D}_A;X)}\leq C(\omega)\abs{t-s}^\nu \leq C\abs{t-s}^\nu
				\end{align}
				for every $t,s\in \R, \omega\in \Omega$ and some mapping $C:\Omega\to [0,\infty)$ which is uniformly bounded.
                \item [iv)] There exists  $\nu^*>0$ such that the adjoint operators $(A^*(t,\omega))_{t\in \R,\,\omega\in \Omega}$ satisfy~\eqref{kt} with exponent $\nu^*$. 
				\item[v)] The mapping $A:\R\times \Omega\to \mathcal{L}(\mathcal{D}_A;X)$ is strongly measurable, adapted and  for every $t\in \R, \omega\in \Omega$ the operator $A(t,\omega)$ has a compact inverse. 
                \item[vi)] For all $t\in \R$ and $\omega\in \Omega$ the structural assumption $A(t,\omega)=A(\theta_t \omega)$ holds.
			\end{itemize}
		\end{assumptions}
        The Assumptions \textit{ii)} and \textit{iii)}  are sometimes called the Kato-Tanabe conditions, see \cite{Tanabe1979}, which are common assumptions for non-autonomous evolution equations. They ensure that the operators are sectorial and satisfying a certain H\"older continuity in time.~{The uniform boundedness of the random constant $C(\omega)$ in Assumption \textit{iii)} can be removed by a localization argument.~We refer to~\cite[Section 5.3]{Pronk2014} for more details on this procedure.} The last two assumptions are necessary because the generators depend on the random parameter $\omega$. In particular, the structural dependence in \textit{vi)} is essential to show that the problem generates a random dynamical system.\\ 
        
 
		In the following, we use the notation $\Delta:=\{(t,s) \in \R^2~:~ t\geq s\}$. The next theorem provides the existence of a random parabolic evolution family generated by $(A(t,\omega))_{t\in \R,\,\omega\in \Omega}$.
  
\begin{theorem}{\rm (\cite[Theorem 2.2, Proposition 2.4]{Pronk2014})} \label{thm:ExistEvolutionFam}
Under the Assumption \ref{ass:OperatorFamily} there exists a unique parabolic evolution family $U:\Delta\times \Omega \to \mathcal{L}(X)$ which is strongly measurable in the operator topology and satisfies the following properties:
			\begin{itemize}
				\itemsep -4pt
				\item[i)] $U(t,t,\omega)=\textnormal{Id}_X$ for all $t\geq 0, \omega\in \Omega$.
				\item[ii)] For all $0\leq r\leq s\leq t,\omega\in \Omega$ we have
				\begin{align}\label{eq:OpFamProp}
					U(t,s,\omega)U(s,r,\omega)=U(t,r,\omega).
				\end{align}
			\item[iii)] For all $s<t$ 
			\begin{align*}
				\frac{\txtd }{\txtd t} U(t,s,\omega)=A(t,\omega)U(t,s,\omega)
			\end{align*}
			holds pointwise in $\Omega$.
			\item[iv)] $U(\cdot,\cdot,\omega)$ is strongly continuous for every $\omega\in \Omega$, and $U(t,s,\cdot)$ strongly $\mathcal{F}_t$-measurable 
            in the uniform operator topology for every $t\geq s$.
			\end{itemize}
		\end{theorem}
We need additional smoothing properties of the evolution family. To this end for $\beta\geq 0$ we denote  the fractional power spaces by $X_\beta:=D((-A(t,\omega))^\beta)$ endowed with the norm $\norm{x}_{X_\beta}:=\norm{(-A(t,\omega))^\beta x}_X$ and by $X_{-\beta}$ the completion of $X$ with respect to the norm $\norm{x}_{X_{-\beta}}:=\norm{(-A(t,\omega))^{-\beta} x}_X$. Since we assumed that $A(t,\omega)$ is densely defined, all fractional powers $A(t,\omega)^\beta$ are also closed and densely defined, see for example \cite[Theorem 4.6.5]{Amann1995}.

    \begin{remark}\label{rem:EquivNormFracPower}
        We impose in Assumption \ref{ass:OperatorFamily} i) that the domain of $A(\theta_t\omega)$ is independent of $t$ and $\omega$.~However, this, in general, does not imply that the fractional power spaces are independent of $t$ and $\omega$. However, since we further assume that the operators $A(\theta_t \omega)$ have bounded imaginary powers, the fractional power spaces can be identified using complex interpolation, meaning that for any $\alpha\in(0,1)$ we have according to \cite[Theorem V.1.5.4]{Amann1995} that $X_\alpha=[X,\cD_A]_\alpha=D((-A(\theta_t \omega))^\alpha)$.~Therefore, the spaces $X_\alpha$ do not depend on $t$ and $\omega$.
		\end{remark}
		  
\begin{assumptions}\label{ass:ExpStable}
The evolution family is exponentially stable. This means that there exists $\lambda>0, C_U>0$ such that 
\begin{align}\label{eq:ExpStable}
	\norm{U(t,s,\omega)}_{\mathcal{L}(X)}\leq C_U e^{-\lambda (t-s)},
				\end{align}
                for all $t\geq s$ and $\omega\in \Omega$.
		\end{assumptions}
     
Assumption \ref{ass:ExpStable} implies the following decay estimates. 

\begin{lemma}{\rm (\cite[Lemma 2.6, 2.7]{Pronk2014})}\label{lem:regularity_process}
			Let the Assumptions \ref{ass:OperatorFamily} and \ref{ass:ExpStable} hold. Then for every $t>0$ and $\omega\in \Omega$ the mapping $U(t,\cdot,\omega)\in C^1([0,t);\mathcal{L}(X))$ fulfills for all $x\in \mathcal{D}_A$
			\begin{align*}
				\frac{\txtd }{\txtd s} U(t,s,\omega)=-U(t,s,\omega)A(s,\omega)x.
			\end{align*}
			Moreover, for $\alpha\in [0,1]$ and $\beta\in (0,1)$ the following estimates hold for $t>s$ and  $\omega\in \Omega$
			\begin{align}
				\norm{(-A(t,\omega))^\alpha U(t,s,\omega)x}_X&\leq \tilde{C}_\alpha \frac{e^{-\lambda(t-s)}}{(t-s)^\alpha}\norm{x}_X, &x\in X,\label{eq:EvFamIneq}\\
				\norm{ U(t,s,\omega)(-A(s,\omega))^\alpha x}_X&\leq \tilde{C}_\alpha \frac{e^{-\lambda(t-s)}}{(t-s)^\alpha}\norm{x}_X, &x\in X_\alpha,\label{eq:EvFamIneq2}\\
				\norm{(-A(t,\omega))^{-\alpha} U(t,s,\omega)(-A(s,\omega))^\beta x}_X&\leq \tilde{C}_{\alpha,\beta} \frac{e^{-\lambda(t-s)}}{(t-s)^{\beta-\alpha}}\norm{x}_X, &x\in X_\beta.\label{eq:EvFamIneq3}
			\end{align}
		\end{lemma}

		\begin{remark}\label{remark}\hfill
            \begin{itemize}
                \item 
    			Note that for $\alpha\in [0,1]$, the space $X_\alpha$ is densely embedded into $X$. Therefore,  \eqref{eq:EvFamIneq2} shows that there exists a unique extension of $U(t,s,\omega)(-A(s,\omega))^\alpha$ as a bounded linear operator on $X$.~We will denote the extension again by $U(t,s,\omega)(-A(s,\omega))^\alpha$. In particular, we have for $x\in X$
    			\begin{align}
    				\norm{U(t,s,\omega)(-A(s,\omega))^{\alpha}x}_X\leq \tilde{C}_\alpha \frac{e^{-\lambda(t-s)}}{(t-s)^\alpha}\norm{x}_X.\label{eq:EvFamIneq4}
     			\end{align}

                \item
                Similarly, we can extend $U(t,s,\omega)$ to an operator in $X_{-\alpha}$. Indeed,  $A(t,\omega)^\alpha$ is an isomorphism between $X_\alpha$ and $X$. Hence, for $x\in X$ there exists an $y\in X_\alpha$ such that $x=A(t,\omega)^\alpha y$, see \cite[Theorem V.1.5.4]{Amann1995}. This implies that
            \begin{align*}
                \|U(t,s,\omega)x\|_{X} &= \|U(t,s,\omega)(-A(t,\omega))^{\alpha}y\|_{X} \le \tilde{C}_\alpha\frac{e^{-\lambda(t-s)}}{(t-s)^{\alpha}}\|y\|_{X} \\
                &= \tilde{C}_\alpha\frac{e^{-\lambda(t-s)}}{(t-s)^{\alpha}}\|(-A(t,\omega))^{-\alpha}x\|_{X} = \tilde{C}_\alpha\frac{e^{-\lambda(t-s)}}{(t-s)^{\alpha}}\|x\|_{X_{-\alpha}},
            \end{align*}
            and since $X$ is densely embedded into $X_{-\alpha}$, we can extend $U(t,s,\omega)$ to an operator from $X_{-\alpha}$ to $X$.
            \end{itemize}

		\end{remark}
  
   Keeping this in mind, the mild formulation corresponding to the linear part of \eqref{eq:MainEq} would be 
        \begin{align*}
    u(t)=U(t,0,\omega)u_0(\omega)+ \sigma \int_0^t U(t,s,\omega)~\txtd W_s\quad {\text{for every $t>0$, } \overline{\P}-\text{a.s.}}
        \end{align*}
        However, the stochastic integral is not well-defined since $U(t,s,\cdot)$ is not $\cF_s$-measurable. The idea to overcome this problem is to formally apply integration by parts~\cite{Pronk2014}. This leads to the integral representation 
        \begin{align}\label{eq:pw_mildLin}
            u(t)&= U(t,0,\omega)u_0 + \sigma U(t,0,\omega)W_t  - \sigma \int_0^t U(t,s,\omega) A(s,\omega)(W_t-W_s)~\txtd s.
            \end{align}
        A process satisfiying \eqref{eq:pw_mildLin} is called a pathwise mild solution.~Note that the stochastic integral in \eqref{eq:pw_mildLin} is well-defined due to the H\"older-regularity of the Brownian motion $(W_t)_{t\geq 0}$  which compensates for the singularity arising from the estimate $\|U(t,s,\omega)A(s,\omega)\|_{\mathcal{L}(X)}\leq (t-s)^{-1}$.~To shorten notations we set 
        \begin{align}\label{eq:def_h}
            h(t):=\sigma U(t,0,\omega)W_t -\sigma \int_0^t U(t,s,\omega)A(s,\omega)(W_t-W_s)~\txtd s.
        \end{align}
{For the sake of completeness, in the following theorem we recall that the linear part of~\eqref{eq:MainEq} 
generates a random dynamical system as shown in~\cite{kuehn2021random}. To this end, we work with the metric dynamical system $(\Omega,\cF,\P,(\theta_t)_{t\in\R})$ constructed in Remark \ref{mds:final} and use the identification $W_t(\omega)=\omega_t$ for $\omega\in\Omega$. Furthermore, the structural assumption on the generators, $A(t,\omega)=A(\theta_t \omega),$ allows us to prove that the corresponding evolution family generates a random dynamical system.}
        \begin{theorem}{\rm (\cite[Theorem 3.1]{kuehn2021random})}\label{lin:rds}
        Let the Assumption~\ref{ass:OperatorFamily} hold and
            $U:\Delta\times \Omega\to \mathcal{L}(X)$ be the evolution family generated by $(A(\theta_t \omega))_{t\in \R,\,\omega\in \Omega}$. Then 
            \begin{align*}
                \widetilde{U}:\R^+\times \Omega\times X \to X,\quad  (t,\omega,x)\mapsto U(t,0,\omega)x,
            \end{align*}
            is a random dynamical system.
        \end{theorem}
        \begin{proof}
            We only give a short argument to emphasize why the structural assumption on the random generators is required. We observe that $$U(t+s,s,\omega)=U(t,0,\theta_s \omega)$$ for $t, s\geq 0$ { and $\omega\in\Omega$}. Intuitively, this means that starting at time $s$ on the $\omega$-fiber, and letting the system evolve for time $t$, is the same as starting at time $0$ on the shifted $\theta_s \omega$-fiber and letting  $t$ time units pass. On the level of generators, the evolution operator $U(t+s,s,\omega)$ is obtained from $A(\theta_t \omega)$ and the evolution operator $U(t,0,\theta_s \omega)$ from $A(\theta_{t-s}\circ \theta_s \omega)$. Due to the group property of the metric dynamical system $\theta$, these generators are the same. Together with \eqref{eq:OpFamProp}, this results in
            \begin{align*}
                \widetilde{U}(t+s,\omega)=\widetilde{U}(t,\theta_s \omega)\widetilde{U}(s,\omega){\text{ for every } t,s\geq 0 \text{ and } \omega\in\Omega.}
            \end{align*}
            For more details we refer to \cite[Theorem 3.1]{kuehn2021random}.
        \end{proof}

\section{Abstract setting {and an Ornstein-Uhlenbeck-type process}}
   
	\subsection{Local existence of solutions}\label{sec:loc:sol}


		The local-in-time existence of a pathwise mild solution for~\eqref{eq:MainEq} can be shown following the proof of~\cite[Theorem 7.8]{Hairer} which relies on Banach's fixed point theorem and subtracting the modified stochastic convolution $h$ in \eqref{eq:def_h}.  This argument is simpler than the setting in~\cite[Theorem 5.3]{Pronk2014} due to the  additive structure of the noise. In \cite{Pronk2014} more general noise terms are considered, but the nonlinearity is globally Lipschitz which we do not assume here.~Since we are interested in the local well-posedness of~\eqref{eq:MainEq}, we work in this subsection with the filtered probability space  $(\overline{\Omega}, \overline{\cF},\overline{\P},(\cF_t)_{t\geq 0} )$.

  \begin{definition}\label{loc:sol}
  Let $\alpha\in[0,1)$.~We call an $X_\alpha$-valued process $u$ together with a stopping time $\tau$ such that $\mathbb{P}(\tau>0)=1$ a local pathwise mild solution for~\eqref{eq:MainEq} if for every $t>0$ the identity 
  \begin{align*}
 u(t)&= U(t,0,\omega)u_0(\omega) + \sigma U(t,0,\omega)W_t + \int_0^t U(t,s,\omega) {(F(u(s,\omega)) + f)}~\txtd s \\
            &\quad - \sigma \int_0^t U(t,s,\omega) A(s,\omega)(W_t-W_s)~\txtd s
     \end{align*}     
  holds {$\overline{\P}$-a.s.}~on the event $\{t<\tau\}$.
Moreover, if $\overline{\tau}$ is a stopping time such that $\overline{\tau}\leq \tau$ a.s. and the corresponding solution satisfies $\overline{u}(t)=u(t)$ a.s. on $\{ t<\overline{\tau}\}$, then the pair $(\overline{u},\overline{\tau})$ is also a local solution and $(u,\tau)$ extends $(\overline{u},\overline{\tau})$.
Finally, we call $(\overline{u},\overline{\tau})$ a maximal mild solution if on the set $\{{\overline{\tau}}<\infty\}$ one has $\lim\limits_{t\to\overline{\tau}} \|u_t\|_{X_\alpha}=\infty$ a.s. and there exists a sequence $(u_n,\tau_n)$ of local mild solutions with increasing stopping times $\tau_n$ such that $\overline{\tau}=\sup\limits_n \tau_n$ a.s. and $(\overline{u},\overline{\tau})$ extends each of the local mild solutions $(u_n,\tau_n)$. {The pathwise mild solution is global provided that $\tau=\infty$ almost surely. }

  \end{definition}

  \begin{theorem}\label{local:existence}
Let $\alpha\in[0,1/2)$, {$f\in X_\alpha$} and let  Assumption~\ref{ass:OperatorFamily} hold. Assume that $u_0(\omega)\in X_\alpha$ is {$\cF_0$-measurable} and that the modified stochastic convolution $h$ in \eqref{eq:def_h} has {$\overline{\P}$-a.s.~continuous} sample paths in $X_\alpha$. Furthermore, let $\beta\geq 0$ be such that $\alpha+\beta<1$ and assume that the map 
\begin{equation}\label{assumption}
    F: X_\alpha \to X_{-\beta},
\end{equation}
is Lipschitz continuous on bounded subsets of $X_\alpha$ and grows at most polynomially, i.e. there exist constants  $R>0,L=L_R>0$, $C_F>0$ and $n_\alpha\geq 1$, such that
\begin{align*}
  &  \|F(u)-F(v)\|_{X_{-\beta}} \leq L \|u - v\|_{X_\alpha},~~\|u\|_{X_\alpha}, \|v\|_{X_\alpha}\leq R;\\
  & \|F(u)\|_{X_{-\beta}} \leq C_F(1+\|u\|^{n_\alpha}_{X_\alpha}), ~~~\quad\quad u\in X_\alpha. 
\end{align*}
Then~\eqref{eq:MainEq}  has a unique maximal pathwise mild solution $u\in C([0,\tau);X_\alpha)$.~Moreover, this solution exists globally provided that $\|u(t)\|_{X_\alpha}\leq C$ for all $t\geq 0$. 
\end{theorem}

\begin{proof}
 We fix a sufficiently small terminal time $T>0$ and apply  Banach's fixed point theorem to the mapping $\Phi:C([0,T];X_\alpha)\to C([0,T];X_\alpha)$ given by
		\[ (\Phi(u))(t) :=\int_0^t U(t,s,\omega){(F(u(s)) + f)}~\txtd s + g(t),~ t\leq T.\]
Here,
	\begin{align*}
  g{(t,\omega)}&= U(t,0,\omega)u_0(\omega) + \sigma U(t,0,\omega)W_t -\sigma\int_0^t U(t,r,\omega)A(r,\omega)(W_t-W_r)~\txtd r\\ & = U(t,0,\omega)u_0(\omega) + h{(t,\omega)}
  \end{align*}
  has~$\overline{\P}$-a.s.~continuous paths in $X_\alpha$ due to our assumptions.~Therefore, $g(\cdot,\omega)\in C([0,T];X_\alpha)$ for all $\omega\in{\overline{\Omega}}$ up to a nullset $N_0$. Let $c:{\overline{\Omega}}\times [0,T]\to \R_+$ be such that
\begin{align*}
    c(t,\omega):=\begin{cases}
        & \sup\limits_{s\in[0,t]}\|g(s,\omega)\|_{X_\alpha}, ~~\omega\in {\overline{\Omega}}\setminus N_0\\
        & 0, \hspace*{30 mm}\omega\in N_0.
    \end{cases}
\end{align*}
We observe that $c(\cdot, \omega)$ is continuous and increasing in $t$.~Furthermore, we define $d:{\overline{\Omega}}\times [0,T]\to\R_+ $ by
\[  
d(t,\omega):=
\max\left\{\tilde{C}_{\alpha,\beta} {L_{c(t,\omega)+1}}, \tilde{C}_{\alpha,\beta}{ C_F {(1+(1+c(t,\omega))^{n_\alpha}) +\|f\|_{X_\alpha} )}}\right\}, 
\]
 where $\|U(t,s,\omega)\|_{\cL(X_{-\beta};X_\alpha)}\leq \tilde{C}_{\alpha,\beta} (t-s)^{-(\alpha+\beta)}$ and the constant $\tilde{C}_{\alpha,\beta}$ is uniformly bounded in $\omega$ by~\eqref{eq:EvFamIneq3} and Assumption~\ref{ass:OperatorFamily} \textit{iii)}.~By definition, $d(t,\cdot)$ is $\cF_t$-measurable for every $t\in[0,T]$. 
 Therefore, {$\tau:{\overline{\Omega}}\to[0,T]$} defined as
\[ \tau(\omega)=\inf\left\{ t\in[0,T] : t^{1-(\alpha+\beta)} d(t,\omega)\geq 1/2 \right\} ,\]
is a stopping time, where we set $\inf \emptyset:=T$.\\

Let $\omega\in{\overline{\Omega}}\setminus N_0$ and $t\leq \tau(\omega)$ which implies that $t^{1-(\alpha+\beta)} d(t,\omega)<1/2 $. We further consider the set 
\[ B:=\left\{ u\in C([0,t];X_\alpha) : \sup\limits_{s\in [0,t]} \|u(s)-g(s)\|_{X_\alpha}\leq 1 \right\}  \]
  and show that $\Phi:C([0,t];X_\alpha)\to C([0,t];X_\alpha)$ leaves  $B$ invariant and is a contraction. To verify the invariance, let $u\in B$. Then $\sup\limits_{s\in[0,t]} \|u(s)\|_{X_\alpha}\leq 1 +c(t,\omega)$ and 
  \begin{align*}
   &\quad \sup\limits_{s\in[0,t]}   \| (\Phi(u)(s)) - g(s,\omega)\|_{X_\alpha} \leq  \sup\limits_{s\in[0,t]} \int_0^ s \|U(s,r,\omega)\|_{\cL(X_{-\beta};X_\alpha)} \|(F(u(r))+f)\|_{X_{-\beta}}~\txtd r
      \\
      & \leq C_F \tilde{C}_{\alpha,\beta} \sup\limits_{s\in[0,t]} \int_0^s (s-r)^{-(\alpha+\beta)} e^{-\lambda(s-r)} \left(1 + \|u(r)\|^{n_\alpha}_{X_\alpha} +\|f\|_{X_\alpha}\right)~\txtd r  \\
      & \leq C_F \tilde{C}_{\alpha,\beta}  \sup\limits_{s\in[0,t]}  \int_0^s (s-r)^{-(\alpha+\beta)} {\left(1+\|u(r)\|_{X_\alpha}^{n_\alpha} {+\|f\|_{X_\alpha}}\right)}~\txtd r\\
      & \leq C_F \tilde{C}_{\alpha,\beta}  t^{1-(\alpha+\beta)}{\left(1+(1+c(t,\omega))^{n_\alpha}+\|f\|_{X_\alpha}\right)} <1/2, 
  \end{align*}
which shows the {invariance of $B$}. 
  
	To show that $\Phi$ is a contraction we use the local Lipschitz continuity of $F$ and obtain  for $u,v\in B$ 
      \begin{align*}
          \| (\Phi(u))(t) - (\Phi(v))(t)\|_{X_\alpha} &\leq 
          \int_0^t \|  U(t,s,\omega)(F(u(s))-F(v(s))\|_{X_\alpha}~\txtd s\\
          & \leq \int_0^t \|U(t,s,\omega)\|_{\cL(X_{-\beta};X_{\alpha})} \|F(u(s))- F(v(s))\|_{X_{-\beta}}~\txtd s\\
          &\leq 
          \tilde{C}_{\alpha,\beta} \int_0^t  (t-s)^{-(\alpha+\beta)} 
          \| F(u(s))-F(v(s))\|_{X_{-\beta}}~\txtd s\\
          &\leq 
          \tilde{C}_{\alpha,\beta} {L_{1+c(t,\omega)}}\int_0^t  (t-s)^{-(\alpha+\beta)} 
          \| u(s)-v(s)\|_{X_{\alpha}}~\txtd s.
      \end{align*}
      
This implies that 
\begin{align*} \sup \limits_{s\in[0,t]} \| (\Phi(u))(s) - (\Phi(v))(s)\|_{X_{\alpha}} &\leq   t^{1-(\alpha+\beta)} d(t,\omega) \sup\limits_{s\in[0,t]} \|u(s)-v(s)\|_{X_\alpha}\\&<\frac{1}{2}\sup\limits_{s\in[0,t]} \|u(s)-v(s)\|_{X_\alpha}, 
\end{align*}
  proving the contraction property.~Therefore, we obtain via Banach's fixed-point theorem a unique pathwise mild solution $(u,\tau)$ which belongs to $X_\alpha$.~Iterating this argument entails the existence of a maximal pathwise mild solution $(u,\tau)$. 
\end{proof}

\begin{remark}
     We apply the abstract setting of Theorem~\ref{local:existence} to reaction-diffusion equations in Section~\ref{main}  with $X=L^{2}(\mathcal{O})$, $\alpha=1/2-\varepsilon$ for $\varepsilon\in(0,1/2)$ and $\beta=s/2$ for $s\in[0,1)$. We will see that the assumptions of Theorem~\ref{local:existence} hold if the {nonlinearity $F$ satisfies} suitable growth assumptions, see Lemma~\ref{lemma:F}.
    
     The global existence, i.e. $\tau=\infty$ a.s.,~will be shown {under additional dissipativity assumptions on the nonlinearity $F$}
        at the level of the random PDE  similar to~\cite[Section 3]{Lu} and~\cite{Caraballo}. Subsequently, we prove that the reaction diffusion equation generates {a random dynamical system} and investigate its asymptotic behavior in Section~\ref{main}.
\end{remark} 

\subsection{Temperedness of an Ornstein-Uhlenbeck-type process}\label{ou}
{From now on, since }{we follow a random dynamical system approach we 
consider} the $(\theta_t)_{t\in \R}$-invariant set 
$\Omega\subset C_0(\R;X)$ obtained in Lemma~\ref{lem:GrowthPropBM} and the metric dynamical system $(\Omega,\cF,\mathbb{P}, (\theta_t)_{t\in\R})$ constructed in Remark~\ref{mds:final}. 


In this subsection we investigate properties of the stationary Ornstein-Uhlenbeck process, i.e. the solution of the following linear problem with random non-autonomous generators 
	    \begin{align}\label{eq:Langevin}
	    	\begin{cases}
		    	\txtd Z(t,\omega)=A(\theta_t \omega)Z(t,\omega)~\txtd t + {\txtd \omega_t}&\\
		    	Z(\omega):=Z(0,\omega)=\int_{-\infty}^0U(0,r,\omega)A(\theta_r \omega)\omega_r~\txtd r,&
	    	\end{cases}
	    \end{align}
		where $U$ is the evolution family corresponding to the operators $(A(\theta_t \omega))_{t\in \R,\,\omega\in \Omega}$. For the initial condition for $Z$ 
        we would expect that $Z(\omega)=\int_{-\infty}^0 U(0,r,\omega)~\txtd \omega_r$, similar to the autonomous case. However, as above, the integral is not well-defined due to the $\omega-$dependence of $U$. 
        Hence, similarly as in the definition of pathwise mild solutions, we formally apply integration by parts and get
		\begin{align}\label{eq:InitValu}
			\int_{-\infty}^0 U(0,r,\omega)~\txtd \omega_r=\lim\limits_{a\to \infty}\Big(-U(0,a,\omega)\omega_{-a}+\int_{-a}^0 U(0,r,\omega)A(\theta_r \omega)\omega_r~\txtd r \Big),
		\end{align}
		where the limit is taken in $X$. The first term converges to $0$, due to the exponential decay of the evolution family and the subexponential growth of $\omega$, i.e.
		\begin{align*}
			\norm{U(0,-a,\omega)\omega_{-a}}_X&\leq C_U e^{-\lambda a}\norm{\omega_{-a}}_X\leq C_U c_\varepsilon e^{-a(\lambda-\varepsilon)}\to 0
		\end{align*}
		as $a\to \infty$ for $\varepsilon$ small enough, see Remark \ref{rem:decaynoise}. The second term in \eqref{eq:InitValu} is  well-defined due to the sublinear growth \eqref{eq:GrowthBM} of $\omega$ which implies for $-r\geq T_0(\varepsilon,\omega)$ that
		\begin{align*}
			\norm{U(0,r,\omega)A(\theta_r \omega)\omega_{r}}_X &\leq \tilde{C}_1\frac{e^{\lambda r}}{-r}\norm{\omega_r}_X\leq \varepsilon e^{\lambda r} \tilde{C}_1.
		\end{align*}
		Hence, the right-hand side of \eqref{eq:InitValu} is well-defined and converges to the $Z(0,\omega)$ given in \eqref{eq:Langevin}. 
		
		The stationary Ornstein-Uhlenbeck process is given by 
	    \begin{align*}
	    	Z(t,\omega)&=\omega_t+\int_{-\infty}^t U(t,r,\omega)A(\theta_r \omega)\omega_r~\txtd r=\omega_t+\int_{-\infty}^0 	U(t,r+t,\omega)A(\theta_{r+t} \omega)\omega_{r+t}~\txtd r\\
	    	&= \omega_t + \int_{-\infty}^0 U(0,r,\theta_t \omega)A(\theta_{r}\circ \theta_t \omega)(\theta_t \omega_r+\omega_t)~\txtd r\\
            &=\int_{-\infty}^0 U(0,r,\theta_t \omega)A(\theta_{r}\circ \theta_t \omega)\theta_t \omega_r~\txtd r.
	    \end{align*}
		Therefore, we consider the process $(t,\omega)\mapsto Z(\theta_t \omega):=Z(t,\omega)$, and with a slight abuse of notation, for $t>0$ we have
	    \begin{align}\label{eq:OUprocess}
	    	Z(t,\omega)=U(t,0,\omega)Z(\omega)+U(t,0,\omega)\omega_t-\int_0^t U(t,r,\omega)A(\theta_r \omega)(\omega_t-\omega_r)~\txtd r,
	    \end{align}
	    i.e. $Z$ is the pathwise mild solution of~\eqref{eq:Langevin}.
		\begin{lemma}\label{lem:OUTemp}
			The Ornstein-Uhlenbeck process $Z$ in \eqref{eq:OUprocess} is stationary and tempered in $X$.
		\end{lemma}
		\begin{proof} For every $t,s\in \R$ and $\omega\in \Omega$ we have
				\begin{align*}
					Z(t+s,\omega)&=\omega_{t+s}+\int_{-\infty}^{t+s} U(t+s,r,\omega)A(\theta_r \omega)\omega_r~\txtd r\\
					&=\theta_s \omega_t+\omega_s+\int_{-\infty}^{t} U(t+s,r+s,\omega)A(\theta_{r+s} \omega)\omega_{r+s}~\txtd r\\
					&=\theta_s \omega_t+\int_{-\infty}^{t} U(t,r,\theta_s \omega)A(\theta_r\circ \theta_s 	\omega)\theta_s \omega_r~\txtd r=Z(t,\theta_s \omega).
				\end{align*}
			For $k\in \N$ and $t,t_1,\ldots,t_k\in \R$ we set $Z_{t_1,\ldots,t_k}:=(Z(t_1,\cdot),\ldots,Z(t_k,\cdot))$. Then for any $A\in \mathcal{B}(X^k)$, the $(\theta_t)_{t\in\R}$-invariance of $\P$ implies that
			\begin{align*}
				\P(Z_{t_1+t,\ldots,t_k+t}\in A)=\P(Z_{t_1,\ldots,t_k}(\theta_t \cdot)\in 	A)=\P(Z_{t_1,\ldots,t_k}\in A),
			\end{align*}
			which shows that $Z$ is stationary.
			
			To verify the temperedness, we use the sufficient condition \eqref{eq:TempCond}.
			The representation of the Ornstein-Uhlenbeck process \eqref{eq:OUprocess}, together with \eqref{eq:ExpStable} leads to
			\begin{align}\label{eq:ProofTempMainIneq}
				\begin{split}
					\E\Big[\sup\limits_{t\in [0,1]} \norm{Z(\theta_t \omega)}_X\Big]&\leq  C_U\E[\norm{Z(\omega)}_X]+C_U\E\Big[\sup\limits_{t\in [0,1]} \norm{\omega_t}_X\Big]\\
					&+\E\Big[\sup\limits_{t\in [0,1]} \norm{\int_{0}^t U(t,r,\omega)A(\theta_r \omega)(\omega_t-\omega_r)~\txtd r}_X\Big].
				\end{split}
			\end{align}
		To estimate the first term we use the same arguments as before. In particular, for $\omega\in \Omega, \varepsilon<\lambda, a>0$ and $\abs{r}\geq T_0(\omega)$,  we have
			\begin{align*}
				\norm{U(0,-a,\omega)\omega_{-a}}_X&\leq C_U e^{-\lambda a}\norm{\omega_{-a}}_X\leq C_U c_\varepsilon e^{-a(\lambda-\varepsilon)},\\
				\norm{U(0,r,\omega)A(\theta_r \omega)\omega_{r}}_X &\leq \tilde{C}_1\frac{e^{\lambda r}}{-r}\norm{\omega_r}_X\leq e^{\lambda r}\varepsilon \tilde{C}_1.
			\end{align*}
			These estimates allow to bound the first term in \eqref{eq:ProofTempMainIneq} in the following way,
			\begin{align*}
				\E[\norm{Z(\omega)}_X]&\leq \lim\limits_{a\to \infty}\Big(\E[\norm{-U(0,a,\omega)\omega_{-a}}_X]+\E\Big[\norm{\int_{-a}^0 U(0,r,\omega)A(\theta_r \omega)\omega_r~\txtd r}_X\Big] \Big)\\
				&\leq \lim\limits_{a\to \infty}\Big( C_U c_\varepsilon e^{-a(\lambda-\varepsilon)}+   \int_{-a}^0 e^{\lambda r}\varepsilon \tilde{C}_1~\txtd r\Big)\\ 
				&=\lim\limits_{a\to \infty}\Big( C_U c_\varepsilon e^{-a(\lambda-\varepsilon)}+ \frac{1-e^{-\lambda a}}{\lambda}\tilde{C}_1\varepsilon\Big)
				=\frac{1}{\lambda} \tilde{C}_1\varepsilon<\infty.
			\end{align*}
			The second term can be estimated by Doob's maximal inequality. In fact, using that $(\omega_t)_{t\in \R}$ is an $X$-valued martingale implies that
			\begin{align*}
				\E\Big[\sup\limits_{t\in [0,1]} \norm{\omega_t}_X\Big]\leq C \E\Big[(\sup\limits_{t\in [0,1]} \norm{\omega_t}_X)^2\Big]\leq 4C \E[\norm{\omega_1}^2_X],
			\end{align*}
			due to the embedding $L^2(\Omega)\hookrightarrow L^1(\Omega)$. Hence, this term is finite due to  \eqref{eq:GammaRadonifying}. The last term in \eqref{eq:ProofTempMainIneq} can be estimated using the Hölder continuity of the noise in Lemma \ref{lem:GrowthPropBM} \textit{ii)} which yields
			\begin{align*}
				&\E\Big[\sup\limits_{t\in [0,1]} \norm{\int_{0}^t U(t,r,\omega)A(\theta_r \omega)(\omega_t-\omega_r)~\txtd r}_X\Big]\\
                    &\leq\tilde{C}_1\E\Big[\sup\limits_{t\in [0,1]} \int_{0}^t\frac{e^{-\lambda(t-r)}}{t-r}\norm{\omega_t-\omega_r}_X ~\txtd r\Big]\\
				&\leq \tilde{C}_1\E\Big[\norm{\omega}_{C^\gamma([0,1];X)}\sup\limits_{t\in [0,1]} \int_{0}^t e^{-\lambda(t-r)} (t-r)^{\gamma-1} ~\txtd r\Big]\\
				&\leq C \tilde{C}_1 \E[\norm{\omega}_{C^\gamma([0,1];X)}]\leq  C \tilde{C}_1 \E[c_1(\omega,\gamma,0,1)]< \infty,
			\end{align*}
			where the integral is finite since $\gamma-1> -1$.
		\end{proof}

We now generalize the result and show that the Ornstein-Uhlenbeck process is tempered in fractional power spaces.
  
		\begin{proposition}\label{cor:HighRegTemp}
   For all $\beta\in [0,\frac{1}{2})$, the stationary Ornstein-Uhlenbeck process $Z:\R\times\Omega\to X_\beta$, $(t,\omega)\mapsto Z(\theta_t\omega)$ is tempered.
		\end{proposition}
		\begin{proof}
			In comparison to the proof of Lemma \ref{lem:OUTemp}, we need to change the interval in the sufficient condition \eqref{eq:TempCond} to take into account the singularity in $0$ of $e^{-\lambda t}t^{-\beta}$. Consequently, we obtain
			\begin{align}\label{eq:ProofTempHighReg}
				\begin{split}
					\E\Big[\sup\limits_{t\in [1,2]} \norm{Z(\theta_t \omega)}_{X_\beta}\Big]&\leq  \E\bigg[\sup\limits_{t\in [1,2]}\norm{U(t,0,\omega)Z(\omega)}_{X_\beta}\bigg]+\E\Big[\sup\limits_{t\in [1,2]} \norm{U(t,0,\omega)\omega_t}_{X_\beta}\Big]\\
					&+\E\Big[\sup\limits_{t\in [1,2]} \norm{\int_{0}^t U(t,r,\omega)A(\theta_r \omega)(\omega_t-\omega_r)~\txtd r}_{X_\beta}\Big].
				\end{split}
			\end{align}
			Since the norm in $X_\beta$ is given by $\norm{(-A(\theta_t \omega))^\beta \cdot}_X$, we can estimate the terms similarly as in Lemma \ref{lem:OUTemp}. In fact, \eqref{eq:EvFamIneq} implies that
			\begin{align*}
				\norm{(-A(\theta_t \omega))^\beta U(t,0,\omega)Z(\omega)}_X\leq \tilde{C}_\beta \frac{e^{-\lambda t}}{t^\beta} \norm{Z(\omega)}_X,
			\end{align*}
			and since $\sup\limits_{t\in [1,2]} e^{-\lambda t} t^{-\beta}=e^{-\lambda}$, the first term in \eqref{eq:ProofTempHighReg} is bounded by
			\begin{align*}
				\E[\sup\limits_{t\in [1,2]}\norm{U(t,0,\omega)Z(\omega)}_{X_\beta}]\leq \tilde{C}_\beta e^{-\lambda}\E[\norm{Z(\omega)}_X]<\infty,
			\end{align*}
			where $\E[\norm{Z(\omega)}_X]<\infty$ as in Lemma \ref{lem:OUTemp}. For the second term in \eqref{eq:ProofTempHighReg} we use again Doob's maximal inequality and the $\gamma$-radonifying property \eqref{eq:GammaRadonifying} to conclude that
			\begin{align*}
				\E[\sup\limits_{t\in [1,2]}\norm{U(t,0,\omega)\omega_t}_{X_\beta}]&\leq \tilde{C}_\beta e^{-\lambda}\E[\sup\limits_{t\in [1,2]}\norm{\omega_t}_X]
				\leq C\tilde{C}_\beta e^{-\lambda}\E[(\sup\limits_{t\in [0,2]}\norm{\omega_t}_X)^2]\\ &\leq 4C\tilde{C}_\beta e^{-\lambda}\E[\norm{\omega_2}^2_X]<\infty.
			\end{align*}
			To estimate the last term in \eqref{eq:ProofTempHighReg}, we recall that we can extend the inequalities \eqref{eq:EvFamIneq}-\eqref{eq:EvFamIneq3} for all $x\in X$, see \eqref{eq:EvFamIneq4}. Therefore, we get 
			\begin{align*}
				\norm{U(t,r,\omega)A(\theta_r \omega)(\omega_t-\omega_r)}_{X_\beta}&\leq\norm{U\left(t,\frac{t+r}{2},\omega\right)}_{\mathcal{L}(X;X_\beta)}\norm{U\left(\frac{t+r}{2},r,\omega\right)A(\theta_r \omega)(\omega_t-\omega_r)}_{X}\\
				&\leq\tilde{C}_\beta \tilde{C}_1 e^{-\lambda\frac{t-r}{2}}\left(\frac{t-r}{2}\right)^{-\beta}e^{-\lambda\frac{t-r}{2}}\left(\frac{t-r}{2}\right)^{-1}\norm{\omega_t-\omega_r}_X\\
				&\leq 2^{1+\beta}\norm{\omega}_{C^\gamma([0,t];X)}\tilde{C}_\beta\tilde{C}_1 \frac{e^{-\lambda(t-r)}}{(t-r)^{1+\beta-\gamma}},
			\end{align*} 
            for $0\leq r\leq t$. 
            Integrating from $r=0$ to $t$, the resulting integral is finite if and only if $\gamma-1-\beta>-1$, so for $\beta<\gamma<\frac{1}{2}$ we have
			\begin{align*}
					\E\Big[\sup\limits_{t\in [1,2]} &\norm{\int_{0}^t U(t,r,\omega)A(\theta_r \omega)(\omega_t-\omega_r)~\txtd r}_{X_\beta}\Big]\\
					&\leq2^{1+\beta}\tilde{C}_\beta\tilde{C}_1\E\Big[\sup\limits_{t\in [1,2]} \norm{\omega}_{C^\gamma([0,t];X)}\int_{0}^t\frac{e^{-\lambda(t-r)}}{(t-r)^{1+\beta-\gamma}} ~\txtd r\Big]\\
					&\leq 2^{1+\beta}\tilde{C}_\beta\tilde{C}_1\E\Big[\norm{\omega}_{C^\gamma([0,2];X)}\sup\limits_{t\in [0,1]} \int_{0}^t \frac{e^{-\lambda(t-r)}}{(t-r)^{1+\beta-\gamma}} ~\txtd r\Big]\\
					&\leq 2^{1+\beta} C \tilde{C}_\beta\tilde{C}_1 \E[\norm{\omega}_{C^\gamma([0,2];X)}]\leq  2^{1+\beta}C \tilde{C}_\beta\tilde{C}_1 \E[c_1(\omega,\gamma,0,2)]< \infty,
			\end{align*}
			where we used again Lemma \ref{lem:GrowthPropBM} \textit{ii)} and the constant $C$ as a bound for the integral. 
		\end{proof}		

  In this framework, we point out the following result of independent interest on the temperedness of $Z$ in more regular spaces than $X_{\beta}$ for $\beta<1/2$.~To this end we assume that the noise is more regular in space, in particular, we consider an $X_\alpha$-valued Brownian motion, and show the temperedness in more regular fractional power spaces. As previously for $X$ it is defined using a $\gamma$-radonifying operator $G:H\to X_\alpha$. Now let $\Omega_\alpha\subset C_0(\R;X_\alpha)$ be the set such that \eqref{eq:GrowthBM} and \eqref{eq:HölderCondBM} are satisfied for an $X_\alpha$-valued Brownian motion 
        endowed with the $\sigma$-algebra $\mathcal{F}_\alpha:=\mathcal{B}(C_0(\R;X_\alpha))\cap \Omega_\alpha$ and let $\P_\alpha$  be the restriction of $\P_W$ to $\mathcal{F}_\alpha$. 
\begin{lemma}
For all $\alpha\in [0,1]$ and $\beta\in [\alpha,\alpha+\gamma)$, the Ornstein-Uhlenbeck process $Z:\R\times\Omega_\alpha\to X_\beta$ is tempered.
\end{lemma}

\begin{proof}
We observe that
\begin{align}\label{eq:ProofTempHighReg2}
\begin{split}
					\E\Big[\sup\limits_{t\in [1,2]} \norm{Z(\theta_t \omega)}_{X_\beta}\Big]&\leq  	\E\bigg[\sup\limits_{t\in [1,2]}\norm{U(t,0,\omega)Z(\omega)}_{X_\beta}\bigg]+\E\Big[\sup\limits_{t\in [1,2]} \norm{U(t,0,\omega)\omega_t}_{X_\beta}\Big]\\
					&+\E\Big[\sup\limits_{t\in [1,2]} \norm{\int_{0}^t U(t,r,\omega)A(\theta_r 	\omega)(\omega_t-\omega_r)~\txtd r}_{X_\beta}\Big].
\end{split}
\end{align}
The first two terms can be estimated exactly as in the proof of Corollary \ref{cor:HighRegTemp} using that $X_\alpha \hookrightarrow X$. Therefore, we only need to estimate the last term. We have for $t>r$
			\begin{align*}
&\norm{U(t,r,\omega)A(\theta_r \omega)(\omega_t-\omega_r)}_{X_\beta}\leq\norm{U\left(t,\frac{t+r}{2},\omega\right)}_{\mathcal{L}(X;X_\beta)}\norm{U\left(\frac{t+r}{2},r,\omega\right)A(\theta_r \omega)(\omega_t-\omega_r)}_{X}\\
				&\leq\tilde{C}_\beta e^{-\lambda\frac{t-r}{2}}\left(\frac{t-r}{2}\right)^{-\beta}\norm{U\left(\frac{t+r}{2},r,\omega\right)(-A(\theta_r \omega))^{1-\alpha}(-A(\theta_r \omega))^{\alpha}(\omega_t-\omega_r)}_{X}\\
				&\leq\tilde{C}_\beta \tilde{C}_{1-\alpha} e^{-\lambda\frac{t-r}{2}}\left(\frac{t-r}{2}\right)^{-\beta}e^{-\lambda\frac{t-r}{2}}\left(\frac{t-r}{2}\right)^{-1+\alpha}\norm{\omega_t-\omega_r}_{X_\alpha}\\
				&\leq 2^{1+\beta}\norm{\omega}_{C^\gamma([0,t];X_\alpha)}\tilde{C}_\beta\tilde{C}_{1-\alpha} \frac{e^{-\lambda(t-r)}}{(t-r)^{1+\beta-\alpha-\gamma}},
			\end{align*}
            where we used that the norms in the fractional power spaces are equivalent for $t\in [1,2]$, see Remark \ref{rem:EquivNormFracPower}. Therefore, all terms on the right-hand side in \eqref{eq:ProofTempHighReg2} are finite, i.e. $Z:\R\times\Omega_\alpha\to X_\beta$ is tempered.
		\end{proof}
	
		\section{Random attractors for reaction-diffusion equations}\label{main}
  
		Let $D\subset\R^N$ be a bounded domain with smooth boundary $\partial D$. We consider the following stochastic reaction-diffusion equation
        \begin{equation}\label{RDE}
            \begin{cases}
            \txtd u  = \bra{\nabla\cdot(\mathcal{E}(x,t,\omega)\nabla u ) + F(u) { + f(x)}}\txtd t + \sigma \txtd W_t(x,\omega), &x\in D,t>0,\omega\in\Omega,\\
            u(x,t,\omega)  = 0, &x\in\partial D, t>0,\omega\in\Omega,\\
            u(x,0,\omega) = u_0(x,\omega), &x\in D,\omega\in\Omega,
            \end{cases}
        \end{equation}
        where {the noise $W_t$} is white-in-time and colored-in-space with a covariance operator as introduced in Section~\ref{sd} with  $X:= L^2(D)$. The random diffusion matrix satisfies the following ellipticity and structural assumptions. 
      
        \begin{enumerate}[label=(\textbf{E\theenumi}),ref=\textbf{E\theenumi}]
        \item\label{A1} $\mathcal E: D\times \R_+ \times \Omega \to\R^{N\times N}_{\text{sym}}$ is {symmetric, bounded, H\"older continuous in $t$, continuously differentiable in $x$ (uniformly w.r.t. the remaining variables), and} uniformly elliptic, i.e. there is $\delta>0$ such that
        \begin{equation*}
            \xi^{\top}\mathcal{E}(x,t,\omega)\xi \ge \delta|\xi|^2 \quad \forall (x,t,\omega)\in D\times \R_+\times \Omega,\; \forall \xi\in \R^N.
        \end{equation*}
      \item\label{A2} For all $(t,\omega)\in \mathbb R\times \Omega$ it holds $\mathcal E(t,\omega) = \mathcal E(\theta_t\omega)$, where $\theta_t \omega(\cdot) = \omega(\cdot+t) - \omega(t)$.

    \end{enumerate}
          For problem \eqref{RDE} we consider $X = L^2(D)$ and denote the  inner product by $\langle\cdot,\cdot\rangle$. The corresponding fractional power spaces are defined through interpolation of $X$ and $X_{1}:= D((-\mathcal E(\cdot,t,\omega))) = H^2(D)\cap H_0^1(D)$. Note that in this case we can identify  $X_{1/2} = H_0^1(D)$ and its dual space as $X_{-1/2} = H^{-1}(D)$. We now define the operator $A(t,\omega): X_{1/2} \to X_{-1/2}$ by
    \begin{equation*}
        \langle A(t,\omega)u, v\rangle = \int_{D}\nabla v\cdot \mathcal{E}(x,t,\omega)\nabla u ~\txtd x, \quad \forall u, v\in X_{1/2}. 
    \end{equation*}
       
    The nonlinearity $F$ satisfies the following growth and dissipativity assumptions.
        \begin{enumerate}[label=(\textbf{F\theenumi}),ref=\textbf{F\theenumi}]
        \item\label{F} The function $F: \R \to \R$ is locally Lipschitz continuous. Moreover, either $N=1$ or $N\ge 2$ and there exists $C_F>0$ such that for all $u, v\in \mathbb R$,
        \begin{equation*}
            |F(u) - F(v)| \le C_F|u - v|\bra{|u|^{\rho-1} + |v|^{\rho-1}},
        \end{equation*}
       where the exponent $\rho$ satisfies
        \begin{equation}\label{rho}
            \rho < \rho_{\text{critical}}:= 
            \left\{
            \begin{aligned}
                &+\infty && \text{ if } N = 2,\\
                &\frac{N}{N-2} &&\text{ if } N\ge 3.
            \end{aligned}
            \right.
        \end{equation}
        \item\label{F1}
   There are positive constants $C_0, C_1$ such that
    \begin{equation}\label{dissi}
    F(u)u \le -C_0|u|^{1+\rho} + C_1 \quad \forall u\in \R
    \end{equation}
    where $\rho<+\infty$ arbitrary if $N=1$ and $\rho$ satisfies \eqref{rho} if $N\ge 2$.
\end{enumerate}
The local Lipschitz continuity and growth assumption in \eqref{F} imply the local existence of pathwise mild solutions, while the dissipativity assumption \eqref{F1} ensures
 global existence.
\\ 
Note that for $N\geq 2$ the assumption (\ref{F}) implies that  $$|F(u)|\leq C_F(1+|u|^\rho)$$ for all $u\in \R$, where we use under slight abuse of notation the same constant as in \eqref{F1}.

\begin{remark}

    {Due to the noise, the global well-posedness of~\eqref{eq:MainEq} and the existence of the random attractor are shown under the growth condition \eqref{rho} which is more restrictive than for deterministic autonomous problems where the subcritical growth assumption is $\rho<\frac{2N}{N-2}$}. This assumption only suffices for the local existence of solutions of~\eqref{RDE}. Moreover, we remark that in ~\cite{liu2017long} even for finite-dimensional noise the more restrictive condition $\rho \le 1 + \frac 4N$ was imposed.
\end{remark}

    \begin{lemma}\label{lemma:F}
        Let the assumption \eqref{F} be satisfied. Then, for any $s\in [0,1)$ and $\varepsilon\in(0,1/2)$ the Nemytskii operator
        $\widetilde{F}$,        \begin{equation*}
            \widetilde{F}(u):= F(u(\cdot)),
        \end{equation*}
        is Lipschitz continuous from $X_{1/2-\varepsilon}$ to $X_{-s/2}$ on bounded subsets of $X_{1/2-\varepsilon}$. 
    \end{lemma}
    

    
    \begin{proof}
        The lemma follows from the proof of \cite[Theorem 12.1]{carvalho2012attractors} so we omit it here.
    \end{proof} 
    We set $\widetilde{W}_t(\omega):= W_t(\cdot,\omega)$, where $W_t(\cdot,\omega)$ is the two-sided Brownian motion defined in Section~\ref{sd}. We recall that $\Omega\subset C_0(\R;X)$ is the $(\theta_t)_{t\in\R}$-invariant subset of full measure obtained in Lemma~\ref{lem:GrowthPropBM}. With a slight abuse of notation, we will write $F$ and $W$ instead of $\widetilde{F}$ and $\widetilde{W}$. The reaction-diffusion equation \eqref{RDE} can be then rewritten in the abstract form
    \begin{equation}\label{specific_equation}
				\txtd u(t,\omega) = (A(\theta_t\omega)u(t,\omega) + F(u(t,\omega)) {+f})~\txtd t + \sigma \txtd W_t(\omega).
			\end{equation}
		Let $Z(t,\omega)$ be { a stationary Ornstein-Uhlenbeck type process}, i.e. the { stationary} solution {of the linear equation}
		\begin{equation}\label{equation_Z}
			\begin{cases}
				\txtd Z(t,\omega) = A(\theta_t\omega)Z(t,\omega)~\txtd t + \txtd W_t(\omega),\\
				Z(\omega):= Z(0,\omega) = \int_{-\infty}^0U(0,r,\omega)A(\theta_r\omega)\omega_r~\txtd r,
			\end{cases}
		\end{equation}  
		and $v(t,\omega)=u(t,\omega) - \sigma Z(\theta_t\omega)$ be the solution of
		\begin{equation}\label{equation_v}
			\begin{cases}
				\dfrac{\txtd}{\txtd t}v(t) = A(\theta_t\omega)v(t) + F(v(t)+\sigma Z(\theta_t\omega)) {+f} ,\\
				v(0,\omega) = v_0(\omega):= u_0(\omega) + \sigma Z(\omega) = u_0(\omega) + \sigma \int_{-\infty}^0U(0,r,\omega)A(\theta_r\omega)\omega_r~\txtd r.
			\end{cases}
		\end{equation}
		Then $u(t,\omega) = v(t,\omega) + \sigma Z(\theta_t\omega)$ is the solution to the original problem \eqref{specific_equation}. 

    \medskip
    
    \medskip
    Combining Proposition \ref{cor:HighRegTemp} with Sobolev embeddings we obtain the following properties of the stationary Ornstein-Uhlenbeck process $Z$.
    \begin{lemma}\label{lem}
     Assume \eqref{A1}-\eqref{A2}. 
        The process $Z: \R\times \Omega \to X_{\beta}$ is tempered for any $\beta \in [0,\frac 12)$. Consequently, $Z$ is tempered in $L^q(D)$ where {$q = +\infty$ if $N=1$, $1\le q < +\infty$ arbitrary if $N=2$, and $1\le q < \frac{2N}{N-2}$ arbitrary if $N\ge 3$.}
    \end{lemma}
    
    \subsection{Global existence}\label{global}
    {For the rest of this section, we consider the phase space $X_\alpha$ where $\alpha$ is chosen such that
    \begin{equation}\label{alpha}
        \frac{N(\rho-1)}{4(\rho+1)} \le \alpha < \min\left\{\frac N4, \frac 12\right\} \quad \text{ and } \quad X_{\alpha} \hookrightarrow \LO{2\rho}.
    \end{equation}
    Thanks to the conditions on $\rho$ specified in \eqref{F} and \eqref{F1}, such a constant $\alpha$ always exists.} 


    Throughout this section $C$ denotes a universal constant which varies from line to line.

    \begin{lemma}\label{lem:local_v}
        {Assume that \eqref{A1}-\eqref{A2} and \eqref{F} hold.
        Fix $\alpha$ satisfying \eqref{alpha} and let $f\in X_\alpha$}.  For any $\omega\in\Omega$ and initial data {$v_0(\omega)\in X_\alpha$} there exists a local mild solution to \eqref{equation_v} in $X_{\alpha}$ on the maximal interval of existence $(0,T_{\max}(\omega))$ in the following sense: the solution $v\in C([0,t];X_\alpha)$ and satisfies for all $t\in (0,T_{\max}(\omega))$ the variation of constants formula
        \begin{equation*}
            v(t) = U(t,0,\omega)v_0 + \int_0^tU(t,s,\omega) {(F(v(s) + \sigma Z(\theta_s\omega)) + f)}~\txtd s. 
        \end{equation*}
        Moreover, for each $t\in (0,T_{\max}(\omega))$, we have the regularity $v(t) \in X_\eta$ for any $\eta \in [\alpha,1)$.
    \end{lemma}
    \begin{proof}
       The local existence of $v$ is standard and follows by a fixed point argument in $C([0,t];X_\alpha)$ for $t\in[0,T_{\text{max}})$. To show the regularity of $v$, we use the smoothing effect of the parabolic evolution family $U$ in Lemma \ref{lem:regularity_process} and the growth assumption on $F$. For $\eta \in [\alpha, 1)$ and any $t\in (0,T_{\max})$ we obtain,  
        \begin{equation*}
            \begin{aligned}
                &\|v(t)\|_{X_\eta} \le \|U(t,0,\omega)v_0\|_{X_\eta} + \int_0^t\|U(t,s,\omega)\baonew{(F(v(s)+\sigma Z(\theta_s\omega)) + f)}\|_{X_\eta}~\txtd s\\
                &\le \tilde{C}_{\eta,\alpha} e^{-\lambda t}{t^{\alpha-\eta}}\|v_0\|_{X_\alpha} + C\int_0^t(t-s)^{-\eta}\|F(v(s) + \sigma Z(\theta_s\omega))\|_{X}~\txtd s + \baonew{C\|f\|_{X_\alpha}\int_0^t(t-s)^{-\eta+\alpha}~\txtd s}\\
                &\le \tilde{C}_{\eta,\alpha} e^{-\lambda t}{t^{\alpha-\eta}}\|v_0\|_{X_\alpha} + CC_F\int_0^t(t-s)^{-\eta}\bra{{1+}\|v(s)\|_{\LO{2\rho}}^{\rho} + \sigma^\rho\|Z(\theta_s\omega)\|_{\LO{2\rho}}^{\rho}}~\txtd s + \baonew{C\|f\|_{X_\alpha}}\\
                &\le \tilde{C}_{\eta,\alpha}{e^{-\lambda t}}{t^{\alpha-\eta}}\|v_0\|_{X_\alpha} + CC_F\bra{\sup_{s\in [0,t]}\left[\|v(s)\|_{X_\alpha}^{\rho} + \sigma^\rho\|Z(\theta_s\omega)\|_{X_\alpha}^{\rho}\right]}\int_0^t(t-s)^{-\eta}~\txtd s + \baonew{C\|f\|_{X_\alpha}}\\
                &<+\infty,
            \end{aligned}
        \end{equation*}
        thanks to the embedding $X_\alpha \hookrightarrow \LO{2\rho}$, together with the fact that $v\in C([0,t];X_\alpha)$ and {$Z(\cdot)\in C([0,t];X_\alpha)$}. The previous computation also entails that $v(t)\in X_{\eta}$ for every $t\in(0,T_{\text{max}}(\omega))$.  
    \end{proof}
        
    \bao{In the following, once $\omega$ is fixed we write $T_{\max}$ instead of $T_{\max}(\omega)$ for simplicity}. Here $T_{\text{max}}=T_{\text{max}}(\omega)$ denotes the maximal existence time as constructed in the proof of Theorem~\ref{local:existence}.

    \begin{prop}\label{RDE_local_existence}
        Let the assumptions in Lemma \ref{lem:local_v} hold. Then for every $\omega\in\Omega$ and any initial data $u_0(\omega)\in X_{\alpha}$, there exists a unique local mild solution $u$ of \eqref{specific_equation} 
        with maximal interval of existence  $(0,T_{\max})$. 
        
    \end{prop}
    \begin{proof}
        The statement immediately follows from Theorem~\ref{local:existence} and Lemma \ref{lem}. 
    \end{proof}


    
Recall that $X_\alpha\hookrightarrow L^{2\rho}(D)$ if $\alpha$ satisfies \eqref{alpha}. Since $u_0(\omega)\in X_\alpha$, it follows that $v_0(\omega)\in X_\alpha$.

\begin{lemma}\label{lem:L2bound}
        {Assume that \eqref{A1}-\eqref{A2}, \eqref{F} and \eqref{F1} hold.
        Fix $\alpha$ satisfying \eqref{alpha} and let $f\in X_\alpha$}. Then  for {every} $\omega\in\Omega$ and initial data {$v_0(\omega)\in X_\alpha$} we have
        \begin{equation*}
            \|v(t)\|_{L^2(D)} \le C(T_{\max}, \bao{\|f\|_{X_\alpha}}) \quad \forall t\in (0,T_{\max}),
        \end{equation*}
        where $C(T_{\max},\bao{\|f\|_{X_\alpha}})$ depends continuously on $T_{\max}$ \bao{and on the norm of the external force $\|f\|_{X_\alpha}$}.
    \end{lemma}
    \begin{remark}
        \bao{For simplicity, in the following we will drop the dependency on $\|v_0(\omega)\|_{L^2(D)}$,  $\|f\|_{X_\alpha}$ in $C(T_{\max}, \|f\|_{X_\alpha})$ and simply write $C(T_{\max})$.}
    \end{remark}
    
    \begin{proof}
    Taking the inner product of \eqref{equation_v} with $v$ in $L^2(D)$ { and using the regularity of $v$ established in Lemma~\ref{lem:local_v}},  the assumptions \eqref{A1} and \eqref{F1} imply that 
    \begin{equation}\label{e1}
    \begin{aligned}
        \frac{1}{2}\frac{\txtd}{\txtd t}\|v\|_{L^2(D)}^2 &= \langle -A(\theta_t\omega)v, v\rangle + \int_{D}F(v + \sigma Z)v~\txtd x + \baonew{\int_D fv ~\txtd x}\\
        &\le -\delta\|v\|_{H_0^1(D)}^2  - C_0\int_{D}|v+\sigma Z|^{1+\rho}~\txtd x + C_1|D| - \sigma \int_{D}F(v+\sigma Z)Z~\txtd x\\
        &\quad + \baonew{\|v\|_{\LO{2}}\|f\|_{\LO{2}}}\\
        &\le -\delta\|v\|_{H_0^1(D)}^2 -C(\rho)\int_{D}|v|^{1+\rho}~\txtd x + C(\rho,\sigma)\int_{D}|Z|^{1+\rho}~\txtd x\\
        &\quad + C_1|D| - \sigma \int_{D}F(v+\sigma Z)Z~\txtd x   \baonew{ + \frac{\lambda_1}{2}\|v\|_{\LO{2}}^2 + \frac{1}{2\lambda_1}\|f\|_{\LO{2}}^2}.
    \end{aligned}
    \end{equation}
    Here, we used the norm $\|u\|_{H_0^1(D)} = \|\nabla u\|_{L^2(D)}$ which is equivalent to the $H^1(D)$ norm due to the homogeneous Dirichlet boundary condition.
    To deal with the last term on the right-hand side, we use the growth condition \eqref{F} and obtain
        \begin{equation*}
        \begin{aligned}
            \int_{D}F(v+\sigma Z)Z~\txtd x &\le C_F\int_{D}|Z|\bra{|v|^{\rho} + |Z|^{\rho} + 1}~\txtd x\\
            &\le C(\rho)\Big(\frac{1}{2}\int_{D}|v|^{\rho+1}~\txtd x + \int_{D}|Z|^{\rho+1}~\txtd x + 1\Big).
        \end{aligned}
        \end{equation*}
        We now use the Poincar\'e inequality $\|v\|_{H_0^1(D)} \ge \lambda_1\|v\|_{L^2(D)}$ \baonew{ and the continuous embeddings $\|f\|_{\LO{2}} \le C\|f\|_{X_\alpha}$} { and $\|Z\|_{L^{\rho+1}}\leq C \|Z\|_{X_\alpha}$ } in \eqref{e1} to get 
        \begin{equation}\label{h0}
            \begin{aligned}
            \frac{\txtd}{\txtd t}\|v\|_{L^2(D)}^2 +  \delta \lambda_1\|v\|_{L^2(D)}^2 + C(\rho)\int_{D}|v|^{\rho+1}~\txtd x& \le C\bra{\int_{D}|Z|^{\rho+1}~\txtd x+1 + \baonew{\|f\|_{X_\alpha}}}\\
            &\le C\bra{\|Z\|_{X_{\alpha}}^{\rho+1}+1}.
            \end{aligned}
        \end{equation}
        This implies that 
        \begin{equation}\label{L2_estimate}
        \begin{aligned}
            \|v(t,\omega,v_0(\omega))\|_{L^2(D)}^2 &\le e^{-\delta \lambda_1 t}\|v_0(\omega)\|_{L^2(D)}^2 + C\bra{\int_0^te^{-\delta\lambda_1(t-s)}\|Z(\theta_s\omega)\|_{X_{\alpha}}^{\rho+1}~\txtd s+1}\\
            &\le \|v_0(\omega)\|_{L^2(D)}^2 + C\bra{\int_0^{T_{\max}}\|Z(\theta_s\omega)\|_{X_{\alpha}}^{\rho+1}~\txtd s + 1 },
        \end{aligned}
        \end{equation}
    which shows the desired estimate (since $Z\in C([0,T_{\text{max}}];X_\alpha)$), and the continuous dependence of the constant $C(T_{\text{max}})$ on $T_{\text{max}}$.
    \end{proof}


\begin{lemma}\label{lem2}
        Let the assumptions in Lemma \ref{lem:L2bound} hold.  Then {for every} $\omega\in\Omega$,  initial data {$v_0(\omega)\in X_\alpha$} and
        for any $\tau\in (0,T_{\max})$ we have
        \begin{equation*}
            \|v(t+\tau)\|_{L^{\rho+1}(D)} \le \bao{C(\tau,T_{\max})} \quad \forall t\in (0,T_{\max}-\tau),
        \end{equation*}
        where \bao{$C(\tau, T_{\max})$} depends continuously on $\tau$ and $T_{\max}$. 
\end{lemma}
    
\begin{proof}
        Fix $\tau > 0$ and consider $t\in (0,T_{\max}-\tau)$. \bao{We integrate \eqref{h0} over} $(t,t+\tau)$ to get
        \begin{equation}\label{h1_1}
            \int_t^{t+\tau}\|v(s)\|_{\LO{\rho+1}}^{\rho+1}~\txtd s \le \|v(t)\|_{\LO{2}}^2 + C\bra{\tau + \int_t^{t+\tau}\|Z(\theta_s\omega)\|_{X_\alpha}^{\rho+1}~\txtd s}.
        \end{equation}
        \bao{Thanks to the regularity of $v(t)$ in Lemma \ref{lem:local_v}, we can multiply} the equation \eqref{equation_v} by $|v|^{\rho-1}v$ to obtain 
        \begin{equation}\label{h2}
        \begin{aligned}
            \frac{1}{\rho+1}\frac{\txtd}{\txtd t}\|v\|_{\LO{\rho+1}}^{\rho+1} + \rho \int_{D}|v|^{\rho-1}\nabla v \cdot \mathcal E(x,t,\omega)\nabla v ~\txtd x\\
            \le \int_{D}F(v + \sigma Z)v|v|^{\rho-1}~\txtd x + \baonew{\int_{D} |f||v|^{\rho}\txtd x}.
        \end{aligned}
        \end{equation}
        We verify that the second term on the left-hand side is well-defined, which justifies the previous multiplication. Indeed, if $N\le 4$, thanks to the regularity of $v\in X_{\eta}$ for arbitrary $\eta\in [\alpha,1)$ obtained in Lemma \ref{lem:local_v}, we have, due to embedding theorems, $v\in \LO{q}$ for arbitrary $q\in [1,\infty)$, and therefore the justification of the second term on the left-hand side is straightforward. Now, if $N>4$, the boundedness of $\mathcal{E}$ and the regularity of $v$ give $|\nabla v| \in X_{\eta-1/2}$. An embedding result yields $|\nabla v|^2 \in \LO{N/(N-2) - \vartheta'}$  for any $0<\vartheta' \le 2/(N-2)$. It remains to show that $|v|^{\rho-1} \in \LO{N/2+\vartheta''}$ for some $\vartheta''>0$, which is achieved provided $\rho < N/(N-4)$. But this always holds thanks to the assumption of $\rho$ in \eqref{rho}. Now, by using \eqref{A1}, the second term on the left-hand side is non-negative. We further exploit the dissipativity of the nonlinearity $F$ to obtain
        \begin{align*}
            &\int_{D}F(v+\sigma Z)v|v|^{\rho-1}~\txtd x\\
            &= \int_{D}F(v+\sigma Z)(v+\sigma Z)|v|^{\rho-1}~\txtd x - \sigma \int_{D}F(v+\sigma Z)Z|v|^{\rho-1}~\txtd x\\
            &\le -C_0\int_{D}|v+\sigma Z|^{\rho+1}|v|^{\rho-1}~\txtd x + C_1\int_{D}|v|^{\rho-1}~\txtd x - \sigma \int_{D}F(v+\sigma Z)Z|v|^{\rho-1}~\txtd x\\
            &\le -C_0\int_{D}|v+\sigma Z|^{\rho+1}|v|^{\rho-1}~\txtd x + C\int_{D}|v|^{\rho+1}~\txtd x+ C +  C_F\int_{D}(|v+\sigma Z|^{\rho}+1)|Z||v|^{\rho-1}~\txtd x\\
            &\le -C_0\int_{D}|v|^{2\rho}~\txtd x + C_0\sigma^{\rho+1}\int_{D}|Z|^{\rho+1}|v|^{\rho-1} + C\int_{D}|v|^{\rho+1}~\txtd x+ C\\
            &\quad + CC_F\int_{D}|v|^{2\rho - 1}|Z|~\txtd x + CC_F\int_{D}|Z|^{\rho+1}|v|^{\rho-1}~\txtd x + C_F\int_{D}|Z||v|^{\rho-1}~\txtd x.
        \end{align*}
        We use Young's inequality to further estimate the terms on the right-hand side,
        \begin{align*}
            C_0\sigma^{\rho+1}\int_{D}|Z|^{\rho+1}|v|^{\rho-1}~\txtd x &\le \frac{C_0}{5}\int_{D}|v|^{2\rho}~\txtd x + C\int_{D}|Z|^{2\rho}~\txtd x,\\
            C\int_{D}|v|^{\rho+1}~\txtd x &\le \frac{C_0}{6}\int_D |v|^{2\rho}~\txtd x + C,\\
            C\int_{D}|v|^{2\rho-1}|Z|~\txtd x &\le \frac{C_0}{6}\int_{D}|v|^{2\rho}~\txtd x + C\int_{D}|Z|^{2\rho}~\txtd x,\\
            C\int_{D}|Z|^{\rho+1}|v|^{\rho-1}~\txtd x &\le \frac{C_0}{6}\int_{D}|v|^{2\rho}~\txtd x + C\int_{D}|Z|^{2\rho}~\txtd x,\\
            C\int_{D}|Z||v|^{\rho-1}~\txtd x &\le \frac{C_0}{6}\int_{D}|v|^{2\rho}~\txtd x + C\int_{D}|Z|^{2\rho}~\txtd x + C
        \end{align*}
        \baonew{and
        \begin{align*}
            \int_{D}|f||v|^{\rho}\txtd x \le \frac{C_0}{6}\int_{D}|v|^{2\rho}\txtd x + C\|f\|_{\LO{2}}^2 \le \frac{C_0}{6}\int_{D}|v|^{2\rho}\txtd x + C\|f\|_{X_\alpha}^2.
        \end{align*}}
       Inserting these estimates into \eqref{h2} it follows that
        \begin{equation}\label{h3}
            \frac{\txtd}{\txtd t}\|v\|_{\LO{\rho+1}}^{\rho+1} \le C\bra{\int_{D}|Z|^{2\rho}~\txtd x + 1 + \baonew{\|f\|_{X_\alpha}^2}} \le C(\|Z\|_{X_\alpha}^{2\rho} + 1).
        \end{equation}
        We now integrate \eqref{h3} over $(s,t+\tau)$ with $s\in (t,t+\tau)$ to obtain
        \begin{equation*}
            \|v(t+\tau)\|_{\LO{\rho+1}}^{\rho+1} \le \|v(s)\|_{\LO{\rho+1}}^{\rho+1} + C\int_{s}^{t+\tau}(\|Z(\theta_r \omega)\|_{X_\alpha}^{2\rho}+1)~\txtd r.
        \end{equation*}
        Integrating this inequality with respect to $s$ over $(t,t+\tau)$ and using \eqref{h1_1} leads
        \begin{equation}\label{h6}
        \begin{aligned}            &\tau\|v(t+\tau)\|_{\LO{\rho+1}}^{\rho+1} \le \int_t^{t+\tau}\|v(s)\|_{\LO{\rho+1}}^{\rho+1}~\txtd s + C\int_{t}^{t+\tau}\int_{s}^{t+\tau}(\|Z(\theta_r\omega)\|_{X_\alpha}^{2\rho}+1)~\txtd r \txtd s\\
            &\le \|v(t)\|_{\LO{2}}^2 + C\bra{\tau+\tau^2 + \tau\int_{t}^{t+\tau}(\|Z(\theta_s\omega)\|_{X_\alpha}^{2\rho} + \|Z(\theta_s\omega)\|_{X_\alpha}^{\rho+1})~\txtd s}\\
            &\le \|v(t)\|_{\LO{2}}^2 + C(\tau)\bra{1+\int_{t}^{t+\tau}\|Z(\theta_s\omega)\|_{X_\alpha}^{2\rho}~\txtd s}\\
            &\le C(T_{\max}) + C(\tau)\bra{1+\int_{0}^{T_{\max}}\|Z(\theta_s\omega)\|_{X_\alpha}^{2\rho}~\txtd s}
        \end{aligned}
        \end{equation}
        where we used that $\rho \ge 1$ and Young's inequality in the last step.
    \end{proof}

  
  For the global existence, we have to establish a bound for $v$ in $X_\eta$ for $\eta\in[\alpha,1)$. This is given in the next lemma. 


  
 \begin{lemma}\label{lem:bound_Xeta}
           Let the assumptions in Lemma \ref{lem:L2bound} hold.  Then {for every} $\omega\in\Omega$,  initial data {$v_0(\omega)\in X_\alpha$}  and any $\eta> 0$ such that $\eta+\alpha < 1$, we have for all $\tau\in (0,T_{\max}/2)$ 
        \begin{equation*}
            \|\bao{v(t+\tau)}\|_{X_\eta} \le C(t+\tau,T_{\max}) \quad \forall t\in (0,T_{\max}-\tau),
        \end{equation*}
        for some constant $C(t+\tau,T_{\max})$ depending continuously on $T_{\max}$ and such that $\lim_{t\to T_{\max}-\tau}C(t+\tau,T_{\max}) < +\infty$.
    \end{lemma}
    \begin{proof}
        We use the mild formulation for {$v$} and the regularization properties of the parabolic evolution family $U$.  We first show that 
        \begin{equation}\label{h4}
            \|F(\bao{v + \sigma Z})\|_{X_{-\alpha}}\le C(\bao{\|v\|_{\LO{\rho+1}}^{\rho} + \sigma^\rho\|Z\|_{\LO{\rho+1}}^{\rho}}+1).
        \end{equation}
        Indeed, for $\varphi \in X_{\alpha}\hookrightarrow L^{2N/(N-4\alpha)}(D)$ 
                we have
        \begin{align*}
            |\langle F(\bao{v+\sigma Z}), \varphi \rangle| &\le \int_{D}|F(\bao{v+\sigma Z})||\varphi|~\txtd x \le C_F\int_{D}(|\bao{v+\sigma Z}|^{\rho}+1)|\varphi|~\txtd x\\
            &\le C C_F\|\varphi\|_{\LO{2N/(N-4\alpha)}}\|\bao{v+\sigma Z}\|_{\LO{2N\rho/(N+4\alpha)}}^{\rho} + C_F\|\varphi\|_{\LO{1}}\\
            &\le C C_F\|\varphi\|_{X_\alpha}\bra{\|\bao{v+\sigma Z}\|_{\LO{2N\rho/(N+4\alpha)}}^{\rho}+1}.
        \end{align*}
        Thanks to the choice of $\alpha$ {specified in~\eqref{alpha}}, we have  $$\|\bao{v+\sigma Z}\|_{\LO{2N\rho/(N+4\alpha)}} \le C\|\bao{v+\sigma Z}\|_{\LO{\rho+1}} \bao{\le C\|v\|_{\LO{\rho+1}} + C\sigma\|Z\|_{\LO{\rho+1}}},$$ and therefore \eqref{h4} follows. Now, we can estimate the solution $v$ using the mild formulation.  For any $t\in (\tau,T_{\max}-\tau)$ we have according to~\eqref{h4} that 
        \begin{equation*} \label{bootstrap_est} \begin{aligned}
             &\|\bao{v}(t+\tau,\omega,\bao{v}_0(\omega))\|_{X_\eta}\\
             & = \|\bao{v}(\tau,\theta_t\omega,\bao{v}(t,\omega,\bao{v}_0))\|_{X_\eta}\\
             &\leq \|U(t+\tau,t,\omega)\bao{v}(t,\omega,\bao{v}_0)\|_{X_\eta} + \int_{t}^{t+\tau} \|U(t+\tau,s,\omega) \baonew{(F(v(s,\omega,v_0) + \sigma Z(\theta_s\omega)) + f)}\|_{X_\eta}~\txtd s \\
             & \leq \|U(t+\tau,t,\omega)\bao{v}(t,\omega,\bao{v}_0)\|_{X_\eta}\\
             &\qquad + \tilde{C}_{\eta,\alpha} \int_t^{t+\tau} (t+\tau-s)^{-\eta-\alpha}e^{-\lambda(t+\tau - s)} \|F(\bao{v(s,\omega,v_0)) + \sigma Z(\theta_s\omega)})\|_{X_{-\alpha}}~\txtd s\\
             &\qquad + \baonew{\tilde{C}_{\eta,\alpha}\int_{t}^{t+\tau}(t+\tau-s)^{-\eta + \alpha}{e^{-\lambda(t+\tau - s)}}\|f\|_{X_\alpha}\txtd s}\\
             & \leq \|U(t+\tau,t,\omega)\bao{v}(t,\omega,\bao{v}_0)\|_{X_\eta}\\
             &\qquad +\tilde{C}_{\eta,\alpha}C_F \int_{t}^{t+\tau} (t+\tau-s)^{-\eta-\alpha} e^{-\lambda(t+\tau - s)}\bra{\bao{\|v(s,\omega,v_0)\|^{\rho}_{L^{\rho+1}(D)} + \sigma^\rho\|Z(\theta_s\omega)\|_{\LO{\rho+1}}^{\rho}}+1} ~\txtd s\\
             &\qquad + \baonew{\tilde{C}_{\eta,\alpha}\|f\|_{X_\alpha}\int_{t}^{t+\tau}(t+\tau-s)^{-\eta+\alpha}{e^{-\lambda(t+\tau - s)}}\txtd s}.
          \end{aligned} \end{equation*}
          We recall that $U$ satisfies $U(t+\tau,t,\omega)= U(\tau,0,\theta_t\omega)$ as established in Theorem~\ref{lin:rds}.  Moreover, the constants $\tilde{C}_\eta=\tilde{C}_\eta(\omega)$, respectively  $\tilde{C}_{\eta,\alpha}=\tilde{C}_{\eta,\alpha}(\omega)$, in the computation below, are uniformly bounded with respect to $\omega\in\Omega$, due to Assumption~\ref{ass:OperatorFamily} \textit{iii)}.  
         To estimate the first term on the right-hand side we use \eqref{eq:EvFamIneq},
          \begin{equation*}
          \begin{aligned}
              \|U(t+\tau,t,\omega)\bao{v}(t,\omega,\bao{v}_0)\|_{X_\eta} &\le \tilde{C}_\eta \frac{e^{-\lambda(t+\tau-t)}}{(t+\tau-t)^{
              \eta
              }}\|\bao{v}(t,\omega,\bao{v}_0)\|_{\LO{2}} \le C_\eta(\tau) \|\bao{v}(t,\omega,\bao{v}_0)\|_{\LO{2}},
            \end{aligned}
          \end{equation*}
          and the second term is bounded by
          \begin{equation*}
          \begin{aligned}
              &\tilde{C}_{\eta,\alpha}C_F \sup_{s\in(\tau,T_{\max})}(\|Z(\theta_s\omega)\|_{\LO{\rho+1}}^{\rho}+\|\bao{v}(s,\omega,\bao{v}_0)\|^\rho_{L^{\rho+1}(D)})\int_t^{t+\tau}(t+\tau-s)^{-\eta-\alpha}e^{-\lambda(t+\tau-s)}~\txtd s \\
              &\quad + \tilde{C}_{\eta,\alpha}C_F \int_t^{t+\tau}(t+\tau-s)^{-\eta-\alpha}e^{-\lambda(t+\tau-s)}~\txtd s\\
              &\le C_\eta(\tau)\bra{\sup_{s\in(\tau,T_{\max})}(\|Z(\theta_s\omega)\|_{X_\alpha}^{\rho}+\|\bao{v}(s,\omega,\bao{v}_0)\|_{L^{\rho+1}(D)}) + 1}\\
              &\le C_\eta(\tau)\bra{C(t+\tau,T_{\max})+\sup_{s\in(\tau,T_{\max})}\|Z(\theta_s\omega)\|_{X_\alpha}^{\rho}}+1,
            \end{aligned}
          \end{equation*}
          where we used that $\eta + \alpha < 1$, the estimate for $v$ in Lemma \ref{lem2} and {the embedding $X_\alpha\hookrightarrow L^{\rho+1}(D)$}. Obviously, we can bound the last term
 by
 \[ \tilde{C}_{\eta,\alpha}\int_t^{t+\tau} (t+\tau -s)^{-\eta+\alpha} e^{-\lambda(t+\tau-s)}~\txtd s \|f\|_{X_\alpha}\leq  C_{\eta,\alpha}(\tau) C(t+\tau, T_{\text{max}})\|f\|_{X_\alpha}. \]
    \end{proof}
    \begin{theorem}
         Let the assumptions in Lemma \ref{lem:L2bound} hold. Then for every $\omega\in\Omega$ and initial data \bao{$v_0(\omega) = u_0(\omega) + \sigma Z(\omega)\in X_\alpha$},  Problem \eqref{equation_v} has a unique global solution in \bao{$C([0,\infty);X_\alpha)$. Consequently, for every $\omega\in\Omega$ and initial data $u_0\in X_\alpha$ Problem \eqref{specific_equation} has a unique global solution in $C([0,\infty);X_\alpha)$.}
    \end{theorem}
    \begin{proof}
        \bao{From Lemma \ref{lem:bound_Xeta}, by choosing $\eta = \alpha$, we conclude the global solvability of \eqref{equation_v}. The global existence of \eqref{specific_equation} follows immediately due to the relation $u(t,\omega) = v(t,\omega) + \sigma Z(\theta_t\omega)$.}
    \end{proof}

    \subsection{Random attractors}\label{sec:attr}

    Thanks to the global existence of solutions proved in the previous section, if the assumptions in Lemma \ref{lem:L2bound} hold,  \eqref{specific_equation} gives rise to a continuous random dynamical system $\phi$ on $X_\alpha$ defined by
        \begin{equation*}
        \begin{aligned}  
            \phi: \R_+\times \Omega\times X_\alpha &\to X_\alpha\\
            (t,\omega,u_0(\omega)) &\mapsto \phi(t,\omega,u_0(\omega)):= u(t,\omega,u_0(\omega)) = v(t,\omega,v_0(\omega)) + \sigma Z(\theta_t\omega),
        \end{aligned}
        \end{equation*}
        where $u_0(\omega)\in X_\alpha$ and $v_0(\omega) = u_0(\omega) - \sigma Z(\omega)$. \\

We now investigate the long-time behavior of $\phi$ establishing the existence of a random attractor. Firstly, in order to obtain an absorbing set for $\phi$ in $L^2(D)$, respectively $L^{\rho+1}(D)$, we need the following result regarding the temperedness of certain random variables.

            \begin{lemma}\label{lem:TempInt}{\em (\cite[Lemma 3.7]{Bessaih2014}).}
                Let $Y$ be a tempered random variable and $\gamma>0$. Then 
                \begin{align*}
                   \omega\mapsto \int_{-\infty}^0 e^{\gamma t} Y(\theta_t\omega)~\txtd t
                \end{align*}
                is tempered.
            \end{lemma}

         We further recall that $\mathcal D(\bao{X_\alpha})$ denotes the collection of tempered random sets in $X_\alpha$. 

        \begin{lemma}\label{lem1}
         Let the assumptions in Lemma \ref{lem:L2bound} hold.
        Then there exists an absorbing set in $L^2(D)$ for the random dynamical system $\phi$. 
    \end{lemma}
    \begin{proof}
        In the estimate \eqref{L2_estimate} in Lemma \ref{lem:L2bound} we can replace $\omega$ by $\theta_{-t}\omega$ and obtain
        \begin{equation}\label{h1}
        \begin{aligned}
            &\|v(t,\tto,v_0(\tto))\|_{\LO{2}}^2\\
            &\le e^{-\delta\lambda_1 t}\|v_0(\tto)\|_{L^2(D)}^2 + C\bra{\int_0^te^{-\delta\lambda_1(t-s)}\|Z(\theta_{s-t}\omega)\|_{\LO{\rho+1}}^{\rho+1}~\txtd s + 1}\\
            &\le e^{-\delta\lambda_1 t}\|v_0(\tto)\|_{L^2(D)}^2 + C\bra{\int_{-\infty}^0e^{\delta\lambda_1\tau}\|Z(\theta_{\tau}\omega)\|_{X_\alpha}^{\rho+1}~\txtd\tau + 1}
        \end{aligned}
        \end{equation}
        where we used the change of variables $\tau = s-t$ and the embedding $X_\alpha \hookrightarrow L^{\rho+1}(D)$ in the last step.

        If $\{B(\omega)\}_{\omega\in\Omega} \in \cD(\bao{X_\alpha})$ 
        and $u_0(\theta_{-t}\omega) \in B(\theta_{-t}\omega)$ we therefore have
        \begin{align*}
            &\|\phi(t,\theta_{-t}\omega,u_0(\theta_{-t}\omega))\|_{L^2(D)}^2 \le 2\|v(t,\theta_{-t}\omega,v_0(\theta_{-t}\omega))\|_{L^2(D)}^2 + 2\sigma^2\|Z(\omega)\|_{L^2(D)}^2\\
            &\le2 e^{-\delta\lambda_1 t}\|v_0(\theta_{-t}\omega)\|_{L^2(D)}^2 + C+C\int_{-\infty}^0e^{\delta\lambda_1\tau}\|Z(\theta_\tau\omega)\|_{X_{\alpha}}^{\rho+1}~\txtd\tau + 2\sigma^2\|Z(\omega)\|_{L^2(D)}^2\\
            &\le 2e^{-\delta\lambda_1t}\|u_0(\theta_{-t}\omega)\|_{L^2(D)}^2 + 2\sigma^2 e^{-\delta\lambda_1t}\|Z(\theta_{-t}\omega)\|_{L^2(D)}^2\\
            &\quad + C\bra{1+\int_{-\infty}^0e^{\delta\lambda_1\tau}\|Z(\theta_\tau\omega)\|_{X_\alpha}^{\rho+1}~\txtd\tau + \|Z(\omega)\|_{L^2(D)}^2}.
        \end{align*}
        Taking the $\limsup$ we obtain
        \begin{align*}
           \limsup\limits_{t\to \infty} \|\phi(t,\theta_{-t}\omega,u_0(\theta_{-t}\omega))\|_{L^2(D)} \le r_2(\omega),
        \end{align*}
        where 
        \begin{equation*}
            r_2(\omega) := \bra{2C\bra{1+\int_{-\infty}^0e^{\delta\lambda_1\tau}\|Z(\theta_\tau\omega)\|_{X_\alpha}^{\rho+1}~\txtd\tau + \|Z(\omega)\|_{L^2(D)}^2}}^{1/2}.
        \end{equation*}
        Here, we used that $u_0(\theta_{-t}\omega)\in B(\theta_{-t}\omega)$ and the temperedness of $\|Z(\omega)\|_{X_\alpha}$ which was established in Proposition \ref{cor:HighRegTemp} and implies the temperedness of $\|Z(\omega)\|_{L^2(D)}$. 
Therefore, the temperedness of $r_2$ follows from Lemma \ref{lem:TempInt} and Remark \ref{rem:temp}. 
        We conclude that there exists $T_2(\omega)>0$ such that for all $t\ge T_2(\omega)$ 
        \begin{equation*}
            \|\phi(t,\theta_{-t}\omega,u_0(\theta_{-t}\omega))\|_{L^2(D)} \le r_2(\omega) \quad \forall t\ge T_2(\omega),
        \end{equation*}
        which completes the proof of this lemma.
    \end{proof}

    \begin{lemma}\label{lem2_1}
     Let the assumptions in Lemma \ref{lem:L2bound} hold.
        Then there exists an absorbing set in $L^{\rho+1}(D)$ for the random dynamical system $\phi$.
    \end{lemma}
    \begin{proof}
    Taking $\tau=1$ in \eqref{h6} and using \eqref{h1} we have
    \begin{equation*}
    \begin{aligned}
        &\|v(t+1,\theta_{-t-1}\omega,v_0(\theta_{-t-1}\omega))\|_{\LO{\rho+1}}^{\rho+1}\\
        &\le \|v(t,\theta_{-t-1}\omega,v_0(\theta_{-t-1}\omega))\|_{\LO{2}}^2 + C\bra{1 + \int_{t}^{t+1}\|Z(\theta_{s-t-1}\omega)\|_{X_\alpha}^{2\rho}~\txtd s}\\
        &\le e^{-\delta\lambda_1 t}\|v_0(\theta_{-t-1}\omega)\|_{L^2(D)}^2 + C\bra{\int_{-\infty}^0e^{\delta\lambda_1\tau}\|Z(\theta_{\tau}\omega)\|_{X_\alpha}^{\rho+1}~\txtd \tau + 1}\\
        &\quad + C\bra{1+e^{\delta\lambda_1}\int_{t}^{t+1}e^{-\delta\lambda_1(t+1-s)}\|Z(\theta_{s-t-1}\omega)\|_{X_\alpha}^{2\rho}~\txtd s}
        \\
        &\le e^{\delta\lambda_1}e^{-\delta\lambda_1 (t+1)}\|v_0(\theta_{-t-1}\omega)\|_{L^2(D)}^2 + C\bra{1+\int_{-\infty}^0e^{\delta\lambda_1\tau}\|Z(\theta_\tau\omega)\|_{X_\alpha}^{2\rho}~\txtd\tau}.
    \end{aligned}
    \end{equation*}
    Hence, for $\{B(\omega)\}_{\omega\in\Omega} \in \cD(\bao{X_\alpha})$ and $u_0(\theta_{-t}\omega) \in B(\theta_{-t}\omega)$ we have for $t\ge 1$
        \begin{align*}
            &\|\phi(t,\theta_{-t}\omega,u_0(\theta_{-t}\omega))\|_{L^{\rho+1}(D)}^{\rho+1} \le C(\rho,\sigma)\bra{\|v(t,\theta_{-t}\omega,v_0(\theta_{-t}\omega))\|_{L^{\rho+1}(D)}^{\rho+1} + \|Z(\omega)\|_{L^{\rho+1}(D)}^{\rho+1}}\\
            &\le C(\rho,\sigma)e^{\delta\lambda_1}e^{-\delta\lambda_1 t}\|v_0(\theta_{-t}\omega)\|_{L^2(D)}^2 + C\bra{1+\int_{-\infty}^0e^{\delta\lambda_1\tau}\|Z(\theta_\tau\omega)\|_{X_\alpha}^{2\rho}\txtd\tau} + \|Z(\omega)\|_{\LO{\rho+1}}^{\rho+1}\\
            &\le C(\rho,\sigma)e^{\delta\lambda_1}\bra{2e^{-\delta\lambda_1t}\|u_0(\theta_{-t}\omega)\|_{L^2(D)}^2 + 2\sigma^2 e^{-\delta\lambda_1t}\|Z(\theta_{-t}\omega)\|_{L^2(D)}^2}\\
            &\quad + C\bra{1+\int_{-\infty}^0e^{\delta\lambda_1\tau}\|Z(\theta_\tau\omega)\|_{X_\alpha}^{2\rho}\txtd\tau + \|Z(\omega)\|_{\LO{\rho+1}}^{\rho+1}}.
        \end{align*}
        Taking the $\limsup$  we obtain
        \begin{align*}
           \limsup\limits_{t\to \infty} \|\phi(t,\theta_{-t}\omega,u_0(\theta_{-t}\omega))\|_{L^{\rho+1}(D)} \le r_\rho(\omega),
        \end{align*}
        where
    \begin{equation*}
        r_{\rho}(\omega):= C^{1/(\rho+1)}\bra{1+\int_{-\infty}^0e^{\delta\lambda_1\tau}\|Z(\theta_\tau\omega)\|_{X_\alpha}^{2\rho}\txtd\tau + \|Z(\omega)\|_{\LO{\rho+1}}^{\rho+1}}^{1/(\rho+1)}.
    \end{equation*}
        Here, we used that $u_0(\theta_{-t}\omega)\in B(\theta_{-t}\omega)$ and the temperedness of $\|Z(\omega)\|_{X_\alpha}$ which implies the temperedness of  $\|Z(\omega)\|_{L^{\rho+1}(D)}$. 
    Similar calculations as for $r_2$ in the proof of Lemma \ref{lem1} show that $r_\rho$ is tempered, i.e. there exists $T_{\rho}(\omega)>0$ such that
    \begin{equation}\label{h6_0}
        \|\phi(t,\theta_{-t}\omega,u_0(\theta_{-t}\omega))\|_{\LO{\rho+1}} \le r_{\rho}(\omega) \quad \forall t\ge T_{\rho}(\omega),
    \end{equation}
    which shows that an absorbing set for $\phi$ in $\LO{\rho+1}$ exists.   
    \end{proof}
    
To prove compactness and conclude the existence of a random attractor, we further establish the existence of an absorbing set for $\phi$ \bao{in $X_\eta$ for some $\eta > \alpha$}.

    \begin{theorem}\label{thm:attractor}
   Let the assumptions in Lemma \ref{lem:L2bound} hold.  Then the random dynamical system $\phi$ generated by \eqref{RDE} possesses a \bao{unique} random attractor $\{\mathcal{A}(\omega)\}_{\omega\in\Omega}$ in $X_{\alpha}$. 
    \end{theorem}
    \begin{proof}
        We {first estimate $v$} {similarly as in the proof of Lemma~\ref{lem:bound_Xeta}} for some $\eta > \alpha,$
        \begin{equation}\label{h6_1}
            \begin{aligned}
             &\|\bao{v}(t+1,\omega,\bao{v_0}(\omega))\|_{X_\eta} = \|\bao{v}(1,\theta_t\omega,\bao{v}(t,\omega,\bao{v_0}(\omega)))\|_{X_\eta}\\
             &\leq \|U(t+1,t,\omega)\bao{v}(t,\omega,\bao{v_0}(\omega))\|_{X_\eta}\\
             &\quad + \int_{t}^{t+1} \|U(t+1,s,\omega)\baonew{(F({v(s,\omega,v_0(\omega)) + \sigma Z(\theta_s\omega)}) + f)}\|_{X_\eta}~\txtd s \\
             & \leq \|U(t+1,t,\omega)\bao{v}(t,\omega,\bao{v_0}(\omega))\|_{X_\eta}\\
             &\quad + {\tilde{C}_{\eta,\alpha}} \int_t^{t+1} (t+1-s)^{-\eta-\alpha}{e^{-\lambda(t+1-s)}}\|F({v(s,\omega,v_0(\omega)) + \sigma Z(\theta_s\omega)})\|_{X_{-\alpha}}~\txtd s \\
             &\quad + \baonew{\tilde{C}_{\eta,\alpha}\int_t^{t+1}(t+1-s)^{-\eta+\alpha} {e^{-\lambda(t+1-s)}}\|f\|_{X_\alpha}\txtd s} \\
             & \leq \|U(t+1,t,\omega)\bao{v}(t,\omega,\bao{v_0}(\omega))\|_{X_\eta} + \baonew{C\|f\|_{X_\alpha}}\\
             &\quad +\tilde{C}_{\eta,\alpha} C_F\int_{t}^{t+1} (t+1-s)^{-\eta-\alpha}{e^{-\lambda(t+1-s)}}\times \\
             &\qquad\qquad\qquad \times \bra{{{\|v(s,\omega,v_0(\omega))\|^\rho_{L^{\rho+1}(D)} + \sigma^\rho \|Z(\theta_s\omega)}\|^{\rho}_{L^{\rho+1}(D)}+1}}~\txtd s .\\
          \end{aligned} \end{equation}
          Recall that $U(t+1,t,\omega)=U(1,0,\theta_t\omega)$ as stated in Theorem~\ref{lin:rds} and that the constants $\tilde{C}_\eta=\tilde{C}_\eta(\omega)$ as well as $\tilde{C}_{\eta,\alpha}=\tilde{C}_{\eta,\alpha}(\omega)$ in the estimates below are uniformly bounded with respect to $\omega\in\Omega$ due to Assumption~\ref{ass:OperatorFamily} \textit{iii)}. 
        We now replace $\omega$ by $\theta_{-t-1}\omega$ and use \eqref{h1} to obtain
        \begin{equation}\label{h7}
            \begin{aligned}
            &\|U(t+1,t,\theta_{-t-1}\omega)\bao{v}(t,\theta_{-t-1}\omega,\bao{v_0}(\theta_{-t-1}\omega))\|_{X_\eta}\\
            &\le \tilde{C}_{\eta}\frac{e^{-\lambda(t+1-t)}}{(t+1-t)^{\eta}}\|\bao{v}(t,\theta_{-t-1}\omega,\bao{v_0}(\theta_{-t-1}\omega))\|_{\LO{2}}\\
            &\le C_\eta\bra{e^{-\delta\lambda_1 t}\|v_0(\theta_{-t-1}\omega)\|_{L^2(D)}^2 + \int_{-\infty}^0e^{\delta\lambda_1\tau}\|Z(\theta_{\tau-1}\omega)\|_{X_\alpha}^{\rho+1}~\txtd\tau + 1}^{1/2}\\
            &\le C_\eta\bigg[2e^{-\delta\lambda_1 t}\|u_0(\theta_{-t-1}\omega)\|_{L^2(D)}^2 + 2\sigma^2e^{-\delta\lambda_1 t}\|Z(\theta_{-1}\omega)\|_{\LO{2}}^2\\
            &\qquad\qquad  + e^{\delta\lambda_1}\int_{-\infty}^0e^{\delta\lambda_1\tau}\|Z(\theta_{\tau}\omega)\|_{X_\alpha}^{\rho+1}~\txtd\tau + 1\bigg]^{1/2} \\
            &\le C\bigg[e^{-\delta\lambda_1(t+1)}\|u_0(\theta_{-t-1}\omega)\|_{\LO{2}}^2 + \int_{-\infty}^0e^{\delta\lambda_1\tau}\|Z(\theta_{\tau}\omega)\|_{X_\alpha}^{\rho+1}~\txtd\tau + 1\bigg]^{1/2},\\
            \end{aligned}
        \end{equation}
     where we chose $t$ large enough such that $2\sigma^2e^{-\delta\lambda_1 t}\|Z(\theta_{-1}\omega)\|_{\LO{2}}^2\le 1$. 
For the third term on the right-hand side of \eqref{h6_1}, we \bao{also replace $\omega$ by $\theta_{-t-1}\omega$ and then} use the absorbing property of $\phi$ in $L^{\rho+1}(D)$ obtained in~\eqref{h6_0}. Hence,  we have for $t\ge T_{\rho} (\omega)$ 
        \begin{equation}\label{h8}
        \begin{aligned}
            &\tilde{C}_{\eta,\alpha} \int_{t}^{t+1} (t+1-s)^{-\eta-\alpha}\bra{\|v(s,\theta_{-t-1}\omega,v_0(\theta_{-t-1}\omega)) + \sigma Z(\theta_{-t-1}\omega))\|^{\rho}_{L^{\rho+1}(D)}+1}~\txtd s\\
            &\le \tilde{C}_{\eta,\alpha}\bra{\int_t^{t+1}(t+1-s)^{-\eta-\alpha}r_{\rho}(\theta_{s-t-1}\omega)^{\rho}~\txtd s + 1}\\
            &\le C_\eta\bra{\int_t^{t+1}(t+1-s)^{-\eta-\alpha}r_{\rho}(\theta_{s-t-1}\omega)^{\rho+1}~\txtd s + \int_t^{t+1}(t+1-s)^{-\eta-\alpha}~\txtd s + 1 }\\
            &\le C_\eta\bra{1+\int_{0}^1\tau^{-\eta-\alpha}r_{\rho}(\theta_{-\tau}\omega)^{\rho+1}~\txtd \tau}.
        \end{aligned}
        \end{equation}
        From \eqref{h6_1}, \eqref{h7} and \eqref{h8} we conclude that for all $t\ge T_{\rho}(\omega)$
        \begin{align*}
            &\|u(t+1,\theta_{-t-1}\omega,u_0(\theta_{-t-1}\omega))\|_{X_\eta}\\
            &\bao{\le \|v(t+1,\theta_{-t-1}\omega,v_0(\theta_{-t-1}\omega))\|_{X_\eta} +\sigma\|Z(\omega)\|_{X_\eta}} \\
            &\le C_\eta e^{-\delta \lambda_1 (t+1)}\|u_0(\theta_{-t-1}\omega)\|_{\LO{2}} + C +C\|f\|_{X_\alpha}+ \bao{\sigma \|Z(\omega)\|_{X_\eta}}\\
            &\quad + C\bra{\int_{-\infty}^0e^{\delta\lambda_1\tau}\|Z(\theta_{\tau}\omega)\|_{X_\alpha}^{\rho+1}~\txtd\tau}^{1/2} + C_\eta\bra{1+\int_0^1\tau^{-\eta-\alpha}r_{\rho}(\theta_{-\tau}\omega)^{\rho+1}~\txtd\tau}\\
            &=: C_\eta e^{-\delta \lambda_1 (t+1)}\|u_0(\theta_{-t-1}\omega)\|_{\LO{2}} + \frac 12 r_{\eta}(\omega),
        \end{align*}
    where 
    \begin{align*}
        r_\eta(\omega):= C + C\bra{\int_{-\infty}^0e^{\delta\lambda_1\tau}\|Z(\theta_{\tau}\omega)\|_{X_\alpha}^{\rho+1}~\txtd\tau}^{1/2} + \bao{\sigma \|Z(\omega)\|_{X_{\eta}}}+C\|f\|_{X_\alpha}\\
            \quad + C_\eta\bra{1+\int_0^1\tau^{-\eta-\alpha}r_{\rho}(\theta_{-\tau}\omega)^{\rho+1}~\txtd\tau}.
    \end{align*}
    Note that $\eta<1/2$. Hence, using the temperedness of $\|Z(\omega)\|_{X_\eta}$ and of $r_\rho$ it follows that $r_\eta$ is tempered.
    Since $u_0(\theta_{-t-1}\omega)\in B(\theta_{-t-1}\omega)$, there exists $T_\eta(\omega)>0$ such that 
    \begin{equation*}
        \|\phi(t,\theta_{-t}\omega,u_0(\theta_{-t}\omega))\|_{X_\eta} \le r_\eta(\omega), 
    \end{equation*}
    for all $t\ge T_\eta(\omega)+1$. Consequently, the RDS $\phi$ possesses an absorbing set in $X_\eta$. Since $\eta>\alpha$, the embedding $X_\eta \hookrightarrow X_\alpha$ is compact, and we conclude that $\phi$ possesses a unique random attractor $\{\mathcal{A}(\omega)\}_{\omega\in\Omega}$ in $X_\alpha$ according to Theorem \ref{thm:attractor_general}.
    \end{proof}

   \subsection{Generalization for higher-order elliptic operators}\label{ho}
Standard examples of operators  satisfying Assumptions~\ref{ass:OperatorFamily} are uniformly elliptic  operators  with Dirichlet, Neumann or Robin boundary conditions. The following example was given in~\cite{kuehn2021random}, see also~\cite{MShen1,book}. 
It generalizes the results of the previous section for higher-order uniformly elliptic random operators. More precisely, for $m\in \mathbb N$, we consider  the following stochastic partial differential equation
\begin{equation}\label{general_elliptic}
\left\{
    \begin{aligned}
    &\mathrm{d}u = \bra{\mathcal{E}_m(\theta_t\omega,x,\txtD)u + F(u) + \baonew{f(x)}}\mathrm{d}t + \sigma \txtd W_t(x,\omega), &&x\in D, \omega\in\Omega,\\
    &u(x,t,\omega) = 0, &&x\in\partial D, \omega\in\Omega,\\
    &u(x,0,\omega) = u_0(x,\omega), &&x\in D, \omega\in\Omega,
\end{aligned}
\right.
\end{equation}
where $D\subset\mathbb{R}^{N}$ is a bounded domain with smooth boundary $\partial D$, and the differential operator is defined as
\begin{equation*} \mathcal{E}_m(\theta_t\omega,x,\txtD)u:=\sum\limits_{|k_1|,|k_2|\leq m}\txtD^{k_1}(a_{k_1,k_2}(\theta_t\omega,x)\txtD^{k_2}u),\qquad \omega\in\Omega,~ x\in D,~t\in\mathbb{R}.
\end{equation*}
We assume that the coefficient functions satisfy the following properties:
	\begin{enumerate}
		\item The coefficients $a_{k_1,k_2}$ are bounded and symmetric. More precisely, there exists a constant $K\geq 1$ such that 
		$$
		|a_{k_1,k_2}(\theta_t\omega,x)|\leq K\qquad \text{for all}\ |k_1|,|k_2|\leq m,\ t\in\mathbb{R}, x\in D, \omega\in\Omega
		$$ 
		and $$a_{k_1,k_2}(\cdot,\cdot)=a_{k_2,k_1}(\cdot,\cdot)~~\mbox{ for }~|k_1|,|k_2|\leq m.
        $$ 
\item        The coefficients form a stochastic process $(t,\omega)\mapsto a_{k_1,k_2}(\theta_{t}\omega,\cdot{})\in C^{m}(\overline{D})$ with H\"older continuous trajectories. This means that there exists  $\nu\in(0,1]$ such that 
					\begin{align*}
		|a_{k_1,k_2}(\theta_{t}\omega,x)-a_{k_1,k_2}(\theta_{s}\omega,x)|\leq \overline{c}_{2}|t-s|^{\nu}\quad\mbox{ for all } t\in\mathbb{R},~ x\in\overline{D},~  |k_1|, |k_2|\leq m,
		\end{align*}
		for some constant $\overline{c}_{2}>0$.~Furthermore, the mapping $t\mapsto \txtD^{k} a_{k_1,k_2}(\theta_t\omega,x)$ 
        is continuous for $|k|,|k_1|,|k_2|\leq m$, $\omega\in\Omega$ and $x\in D$.
			\item The operator $\mathcal{E}_m$ is uniformly elliptic in $D$, i.e. there exists a constant $\bar c>0$ such that 
		$$ 
		\sum\limits_{|k_1|=|k_2|=m}a_{k_1,k_2}(\theta_t\omega,x)\xi_{k_1}\xi_{k_2}\geq \overline{c}|\xi|^{2m}\qquad \mbox{ for all } t\in\mathbb{R},~x\in\overline{D},~ \xi\in\mathbb{R}^{N}.
		$$
			\end{enumerate} 
    We now introduce the random partial differential operator
on the Banach space $X:=L^{2}(D)$.
    ~To this end we define the $L^{2}$-realization $A_2$ of $\mathcal{E}_m(\cdot,\cdot,\txtD)$ as 
	\begin{align*}
	&A_{2}(\theta_t\omega)u:=\mathcal{E}_m(\theta_t\omega,x,\txtD)u\qquad  \mbox{ for }
	u\in \mathcal{D}_A, \mbox{ where }\\
	&\mathcal{D}_A:=D(A_{2}(\theta_t\omega))=H^{2m}(D)\cap H^{m}_{0}(D).\nonumber 
	\end{align*}
  Then $(A_2(\theta_t\omega))_{\omega\in\Omega, t\in\R}$ satisfies Assumption~\ref{ass:OperatorFamily} and therefore generates a parabolic evolution family on $X$ as established in Theorem~\ref{thm:ExistEvolutionFam}. For the nonlinearity, we assume the following: 
    \begin{itemize}
        \item[(i)] $F:\mathbb R\to \mathbb R$ is locally Lipschitz, and either $N\le 2m$ or $N\ge 2m+1$ and there is $C>0$ such that
        \begin{equation*}
            |F(u)-F(v)| \le C_F|u-v|\bra{|u|^{\mu-1} + |v|^{\mu-1}} \quad \forall u, v\in \mathbb R,
        \end{equation*}
        where $\mu$ satisfies
        \begin{equation*}
            \mu < \mu_{\text{critical}}:= \begin{cases}
                +\infty & \text{ if } N \le 2m,\\
                N/(N-2m) & \text{ if } N \ge 2m+1.
            \end{cases}
        \end{equation*}

        \item[(ii)] There are constants $C_0, C_1$ such that
        \begin{equation*}
            F(u)u \le -C_0|u|^{1+\mu} + C_1, \quad \forall u\in \mathbb R
        \end{equation*}
        with $\mu$ in (i).
    \end{itemize}

    Using the same arguments as in the previous section one can show the following result.
    \begin{theorem}
        Suppose that the above assumptions on the differential operator $\mathcal{E}_m$ and the nonlinearity $F$ hold. Fix $\alpha\in (0,1/2)$ such that
        \begin{equation*}
            \frac{N(\mu - 1)}{4(\mu + 1)} \le \alpha < \min\left\{\frac{N}{4}, \frac 12 \right\} \quad \text{ and } \quad X_\alpha \hookrightarrow L^{2\mu}(D).
        \end{equation*}
        Then for any initial data $u_0(\omega)\in X_\alpha$, \baonew{and any $f\in X_\alpha$} the equation \eqref{general_elliptic} has a unique global solution in $X_\alpha$. Moreover, the random dynamical system $\phi$ generated by \eqref{general_elliptic} as
        \begin{equation*}
            \begin{aligned}
                \phi: \mathbb R_+ \times \Omega \times X_\alpha &\to X_\alpha, \quad
            (t,\omega,u_0(\omega)) \mapsto \phi(t,\omega,u_0(\omega)):= u(t,\omega,u_0(\omega))
            \end{aligned}
        \end{equation*}
        possesses a \bao{unique} random attractor $\{\mathcal{A}(\omega)\}_{\omega\in\Omega}$ in $X_\alpha$.
    \end{theorem}

\section*{Conflict of interest statement}
The authors have no conflicts of interest to declare. All co-authors have seen and agree with the contents of the manuscript and there is no financial interest to report. 
\section*{Data availability}
No data was used for the research described in the article.

\end{document}